%% file: main.tex
\documentclass{article}
\usepackage{graphicx} 

\usepackage{amsmath,amssymb,amsthm,fullpage}
\input{macros}

\begin{document}

\title{Moonflowers and code sparsification}

\author{Shachar Lovett\thanks{Supported by Simons Investigator Award \#929894 and NSF award CCF-2425349} \\
UC San Diego \and Raghu Meka\thanks{Supported by NSF EnCORE: Institute for Emerging CORE Methods in Data Science Award \#2217033 and NSF AF: Small Award \#2425350} \\ UCLA \and Yimeng Wang\thanks{Supported by NSF EnCORE: Institute for Emerging CORE Methods in Data Science Award \#2217033 and NSF AF: Small Award \#2425350} \\ UCLA}
\date{\today}

\maketitle

\begin{abstract}
We introduce \emph{moonflowers}, a weaker analogue of sunflowers. A family of sets $S_1,\ldots,S_k$ is a $k$-moonflower if each set $S_i$ contains at least one element that is absent from all the others. 
We study the extremal problem of determining the largest possible size of a family of sets of size at most $w$ that avoids a $k$-moonflower, and obtain optimal bounds.

Using the extremal moonflower bound and Gilmer's entropy lemma, we obtain universe reduction theorem for moonflower-free families. As an application, we revisit the code sparsification problem studied by Brakensiek and Guruswami (STOC 2025). Using our extremal bounds for moonflowers, we get near-optimal bounds for code sparsification while also simplifying the argument. Concretely, we improve the dependence on the block length from poly-logarithmic to logarithmic, which is necessary.
\end{abstract}

\section{Introduction}
\input{introduction}

\section{Preliminaries}\label{sec:prelim}

\input{preliminaries}

\section{Extremal moonflower bound}\label{sec:moonflower}
\input{moonflower_bound}

\section{Universe reduction}\label{sec:universe_reduction}
\input{universe_reduction}

\section{Improved sparsification with a single $\log n$ factor}
\label{sec:sparsification_single_log_n}
\input{newsparsification}

\newpage

\bibliographystyle{alpha}
\bibliography{refs}

\appendix

\section{Proof of Frankl's theorem}\label{sec:appendix}
\input{appendix}

\end{document}

%% file: macros.tex
\usepackage[T1]{fontenc}
\usepackage[margin=1in]{geometry}

\usepackage{latexsym}
\usepackage[small,bf]{caption2}
\usepackage{graphics}
\usepackage{epsfig}
\usepackage{amsopn}
\usepackage{url}
\usepackage{xcolor}
\usepackage{hyperref}
\usepackage{algorithm}
\usepackage{algpseudocode}
\usepackage{graphics}
\usepackage{comment}

\usepackage[shortlabels]{enumitem}

\usepackage{multirow}

\usepackage{titling}

\usepackage{amsmath}
\usepackage{amssymb}
\usepackage{amsthm}
\usepackage{hyperref}   
\usepackage[nameinlink,noabbrev]{cleveref}
\usepackage{aliascnt}

\usepackage{mathtools}
\usepackage{enumitem}
\usepackage{bbm}
\usepackage{turnstile}

\newtheorem{theorem}{Theorem}[section]

\newaliascnt{lemma}{theorem}
\newtheorem{lemma}[lemma]{Lemma}
\aliascntresetthe{lemma}

\newaliascnt{proposition}{theorem}
\newtheorem{proposition}[proposition]{Proposition}
\aliascntresetthe{proposition}

\newaliascnt{fact}{theorem}
\newtheorem{fact}[fact]{Fact}
\aliascntresetthe{fact}

\newaliascnt{claim}{theorem}

\aliascntresetthe{claim}

\newaliascnt{definition}{theorem}
\newtheorem{definition}[definition]{Definition}
\aliascntresetthe{definition}

\newaliascnt{example}{theorem}

\aliascntresetthe{example}

\newaliascnt{corollary}{theorem}
\newtheorem{corollary}[corollary]{Corollary}
\aliascntresetthe{corollary}

\newaliascnt{assumption}{theorem}

\aliascntresetthe{assumption}

\newaliascnt{observation}{theorem}

\aliascntresetthe{observation}

\newaliascnt{conjecture}{theorem}
\newtheorem{conjecture}[conjecture]{Conjecture}
\aliascntresetthe{conjecture}

\newaliascnt{question}{theorem}
\newtheorem{question}[question]{Question}
\aliascntresetthe{question}

\newaliascnt{problem}{theorem}

\aliascntresetthe{problem}

\crefname{theorem}{theorem}{theorems}
\Crefname{theorem}{Theorem}{Theorems}
\crefname{lemma}{lemma}{lemmas}
\Crefname{lemma}{Lemma}{Lemmas}
\crefname{proposition}{proposition}{propositions}
\Crefname{proposition}{Proposition}{Propositions}
\crefname{fact}{fact}{facts}
\Crefname{fact}{Fact}{Facts}
\crefname{claim}{claim}{claims}
\Crefname{claim}{Claim}{Claims}
\crefname{definition}{definition}{definition}
\Crefname{definition}{Definition}{definitions}
\crefname{example}{example}{examples}
\Crefname{example}{Example}{Examples}
\crefname{corollary}{corollary}{corollaries}
\Crefname{corollary}{Corollary}{Corollaries}
\crefname{assumption}{assumption}{assumptions}
\Crefname{assumption}{Assumption}{Assumptions}
\crefname{observation}{observation}{observations}
\Crefname{observation}{Observation}{Observations}
\crefname{conjecture}{conjecture}{conjectures}
\Crefname{conjecture}{Conjecture}{Conjectures}
\crefname{question}{question}{questions}
\Crefname{question}{Question}{Questions}
\crefname{problem}{problem}{problems}
\Crefname{problem}{Problem}{Problems}

\DeclareMathOperator{\NRD}{NRD}
\DeclareMathOperator{\Ham}{Ham}
\DeclareMathOperator{\SPR}{SPR}

\DeclareMathOperator{\MF}{MF}

\DeclareMathOperator{\CL}{CL}

\renewcommand{\epsilon}{\varepsilon}

\DeclareMathOperator{\high}{\mathrm{high}}
\DeclareMathOperator{\low}{\mathrm{low}}

\DeclareMathOperator{\Sym}{Sym}

\newcommand{\rr}{\mathbb{R}}

\newcommand{\poly}{\mathsf{poly}}
\newcommand{\Supp}{\mathsf{Supp}}

\renewcommand{\Pr}{\mathop{\bf Pr\/}}

\newcommand{\calA}{\mathcal{A}}
\newcommand{\calB}{\mathcal{B}}
\newcommand{\calC}{\mathcal{C}}
\newcommand{\calD}{\mathcal{D}}

\newcommand{\calF}{\mathcal{F}}
\newcommand{\calG}{\mathcal{G}}
\newcommand{\calH}{\mathcal{H}}

\newcommand{\calP}{\mathcal{P}}

\newcommand{\calS}{\mathcal{S}}

\newcommand{\calU}{\mathcal{U}}

\newcommand{\E}{\mathbb{E}}

\newcommand{\supp}{\mathrm{supp}}

\newcommand{\shachar}[1]{\textcolor{blue}{(Shachar: {#1})}}

\newcommand{\yimeng}[1]{\textcolor{orange}{(Yimeng: {#1})}}
\newcommand{\ignore}[1]{{}}

%% file: introduction.tex
Extremal combinatorics is a branch of combinatorics that studies how large a finite combinatorial object needs to be in order to guarantee the existence of certain patterns of interest. One such pattern that has attracted the attention of researchers in the past few decades is \emph{sunflowers}. 
A collection of distinct sets $S_1,\ldots,S_k$ is called a \emph{$k$-sunflower} if their pairwise intersections are all the same; in other words, $S_i \cap S_j = S_1 \cap \cdots \cap S_k$ for all $i \neq j$. 

Setting terminology, a set $S$ is called a \emph{$w$-set} if $|S| \leq w$. Back in 1960, Erd\H{o}s and Rado \cite{ER60} proved that any large family of $w$-sets must contain a $k$-sunflower. Specifically, they showed that if $\calF$ is a family of $w$-sets of size $|\calF| \geq w! \cdot (k-1)^w$, then $\calF$ must contain a $k$-sunflower. In the same work, Erd\H{o}s and Rado \cite{ER60} conjectured the bound can be significantly improved.

\begin{conjecture}[Sunflower conjecture \cite{ER60}] Let $k \geq 3$. There exists $c = c(k)$ such that any family of $w$-sets $\calF$ of size $|\calF| \geq c^w$ must contain a $k$-sunflower.
\end{conjecture}

Despite the simplicity of the statement, improving upon the bound in \cite{ER60} turned out to be rather challenging. For almost sixty years after the sunflower conjecture was raised, even for the case of $k = 3$, the best known bound was still of the form $|\calF| \geq w^{O(w)}$ \cite{Kos97, Fuk18}. This long period void of significant progress finally ended with the work \cite{ALWZ21} which  improved the upper bound to $|\calF| \geq (\log w)^w (k \log \log w)^{O(w)}$. Subsequent works \cite{Rao20, BCW21} built upon the result in \cite{ALWZ21} and obtained the following improved bound.

\begin{lemma}[Improved bounds for the sunflower lemma \cite{BCW21}] 
There exists a constant $C > 0$ such that the following holds. Any family of $w$-sets $\calF$ of size $|\calF| \geq (Ck \log w)^w$ must contain a $k$-sunflower.
\end{lemma}

One of the key innovations of \cite{ALWZ21} was the use of a robust probabilistic generalization of sunflowers called \emph{robust sunflowers}. They showed that if a set family $\calF$ is a robust sunflower of the appropriate parameters, then $\calF$ must contain a large sunflower. Using this, they reduced the problem of finding sunflowers to the problem of finding robust sunflowers. Through a careful counting argument, they were able to establish the robust sunflower lemma which leads to the final improved bound for the sunflower lemma. 

It is worth noting that the robust sunflower lemma parameters are known to be tight. Hence, any improvement to the sunflower lemma must come from a different approach. Apart from its application in the improved sunflower lemma, robust sunflowers found various applications in combinatorics and theoretical computer science. In combinatorics, it led to the resolution of the famous Kahn-Kalai conjecture \cite{KK07} by Park and Pham \cite{PP24}. In theoretical computer science, it was used in proving monotone circuit lower bounds \cite{Ros14, CKR22} and lifting theorems in communication complexity \cite{LMM+22}.

In short, robust sunflowers are a generalization of sunflowers for which one can prove tight bounds, and which has diverse applications in combinatorics and theoretical computer science. This motivates the following broad question:
\[
\begin{array}{c} 
\textit{Are there other natural generalizations / variants of sunflowers for which we are able to prove} \\
\textit{ tight bounds and have applications in combinatorics and theoretical computer science?}
\end{array}
\]

\subsection{Moonflowers}
In this paper, we study a combinatorial object which we call \emph{moonflowers}. A family of sets $S_1,\ldots,S_k$ is called a \emph{$k$-moonflower} if each $S_i$ contains at least one element absent from all other sets. Equivalently, there is a set $I$ such that the sets $S_1 \setminus I,\ldots,S_k \setminus I$ are nonempty and pairwise disjoint. Note that unlike for sunflowers, we do not require anything on the intersections $S_i \cap I$.

It is easy to see that a sunflower is either a moonflower or becomes one after removing one set\footnote{More precisely, let $S_1,\ldots,S_k$ be a sunflower and set $K= \cap S_i$. If $K$ is not one of the sets $S_i$ then $S_1,\ldots,S_k$ is a moonflower. If it is, then after removing it the remaining $k-1$ sets form a moonflower.}.
The converse, however, is not true. For example, take any set system $\calF$ and add a new element to each set; this forms a moonflower but in general not a sunflower. 

Moreover, for a fixed set family $\calF$, the gap between the size of the largest sunflower it contains and the size of the largest moonflower it contains can be exponentially large. 
As an example, consider the set family $\calF:= \{S \subseteq [2n]: |S \cap \{2i-1, 2i\}| = 1, \  \forall i\in [n]\}$. Then $|\calF| = 2^n$ and $\calF$ does not contain a $3$-sunflower\footnote{To see this, suppose $S_1, S_2, S_3 \in \calF$ are distinct sets that form a $3$-sunflower. For any $i \in [n]$, at least two of the sets must intersect the same element in $\{2i-1,2i\}$ and hence all three must contain it. This implies that $S_1=S_2=S_3$ and so the sets are not distinct.}. Next, enlarge the universe to size $2^n + 2n$ and add one new distinct element to each set in $\calF$. Denote the resulting family by $\calF'$. Then $\calF'$ still does not contain a $3$-sunflower. However, $\calF'$ is a $2^n$-moonflower.

Just as in the case of sunflowers, we are interested in the following extremal question on moonflowers: 

\begin{question}\label{q:wsfext}
    How large does a family of $w$-sets have to be to guarantee the existence of a $k$-moonflower?
\end{question}
We prove the following theorem.

\begin{theorem}[Extremal bounds on moonflowers]\label{main_wsf_theorem}  Let $\calF$ be a family of $w$-sets which is $k$-moonflower-free. Then 
\[
|\calF| \leq \binom{k+w-1}{w}.
\]
\end{theorem}

In particular, when $w = \Theta (k)$, the above bound simplifies to $|\calF| \leq \exp (O(w))$. We complement this upper bound with a matching lower bound via an explicit construction. First, note that the family of all $w$-sets in $[k+w-2]$ is $k$-moonflower-free and has size ${k+w-2 \choose w}$ which is near optimal. A modification of this construction, found by ChatGPT, gives a tight bound.

\begin{lemma}[Moonflower lower bound]\label{wsf_lower_bound}
    There exists a family of $w$-sets $\calF$ such that $|\calF| = \binom{k+w-1}{w}$ and $\calF$ is $k$-moonflower-free.
\end{lemma}
Consequently, the upper bound in Theorem \ref{main_wsf_theorem} is tight and $\MF(k,w)={k+w-1 \choose w}$. Similar to robust sunflowers, the notion of moonflowers has applications in both combinatorics and theoretical computer science. In combinatorics, it is identical to  \emph{induced matchings} in a bipartite graph. Specifically, given a graph $G = (V, E)$, a set of edges $M \subseteq E$ is called an \emph{induced matching} of $G$ if (i) $E$ is a matching and (ii) there is no edge in $G$ connecting the endpoints of two distinct edges in $M$.

Given a set family $\calF$, we can treat it as a bipartite graph $G_{\calF} = (L_{\calF}, R_{\calF}, E)$ where vertices in $L_{\calF}$ corresponds to sets in $\calF$ and vertices in $R_{\calF}$ corresponds to $\supp (\calF)$. Then, it is not difficult to see that a $k$-moonflower in $\calF$ corresponds exactly to an induced matching of size $k$ in $G_{\calF}$. Because of this equivalence, \Cref{main_wsf_theorem} immediately implies the following graph-theoretic result.

\begin{corollary}[Extremal bounds on induced matchings]
     Let $G = (L, R, E)$ be a left degree at most $w$ bipartite graph with no isolated vertices such that no two vertices in $L$ have the same neighborhood. If $G$ has no induced matchings of size $k$, then
    \[
   |L| \leq \binom{k+w-1}{w}.
    \]
\end{corollary}

\subsection{Universe Reduction}

Let $\calF \subseteq 2^{[n]}$ be a set family over universe $[n]$. In general, the universe can be quite large. For the purpose of efficient computations and storage memory, we would like to identify a \emph{small} set of important elements $I \subseteq [n]$ such that we do not lose too many sets in $\calF$ after reducing the universe to $I$, i.e., the size $\calF_I:= \{S \in \calF: S \subseteq I\}$ is not too small compared to that of $\calF$. In fact, as we will see in the code sparsification application, what we really need is that after puncturing a few coordinates in $I$, the residual set system is small. That is, $\calF_{\Bar{I}}=\{S \setminus I: S \in \calF\}$ should be small.

In general, this task is not possible. For instance, consider the set $\calF = 2^{[n]}$, i.e., the set of all subsets of $[n]$. Then every element in $[n]$ is contained in exactly half of the sets in $\calF$. Leaving out even one element would lose $|\calF| / 2$ of the sets, and we need to puncture most elements to reduce the size of $\calF$ significantly, say to $2^{n/2}$.

Therefore, it would be ideal if we can identify structures over set families that make such universe reduction possible. It turns out moonflower-freeness is one such desirable structure. We first show the existence of a popular element in large moonflower-free set families.

\begin{theorem}[Large moonflower-free set family implies existence of popular element]\label{thm:large_family_size_implies_existence_of_popular_element}
   Let $\calF \subseteq 2^{[n]}$ be a family of $w$-sets which is $k$-moonflower-free. If $|\calF| > {\binom{k+tw}{tw}}^{\Omega((\log t)/t)}$ for some integer $t \geq 2$, then there exists an element $i \in [n]$ such that $i$ is contained in more than $1/2t$ fraction of the sets in $\calF$.
\end{theorem}

The proof of the above relies on our extremal bound for moonflowers and Gilmer's entropy lemma \cite{Gil22}. We further show that if a $w$-set family $\calF$ is $k$-moonflower-free, then we can identify a small set of elements that cover most sets, in the sense that puncturing them results in a small residual set system.

\begin{theorem}[Universe reduction]\label{thm:universe_reduction}
    Let $\calF \subseteq 2^{[n]}$ be a family of $w$-sets which is $k$-moonflower-free. Then for any $\lambda \in (0, 1]$, there exists $I \subseteq [n]$ of size $|I| \leq O(\lambda^{-1}k \log^3 (kw / \lambda))$ satisfying $|\calF_{\bar{I}}| \leq \exp (\lambda w)$.
\end{theorem}

Because of its simplicity and generality, we believe \Cref{thm:universe_reduction} can be an interesting standalone result with broader applications. We complement this result with a lower bound showing it is near optimal in the interesting regime of parameters $k \gg \lambda w \gg 1$ (we justify why this is the interesting regime in \Cref{subsec:universe_reduction_LB}). 

\begin{lemma}[Universe reduction lower bound]\label{lem:universe_reduction_lower_bnd}
    Let $k, w \geq 2$ be positive integers and $\lambda \in (0,1]$. Assume that $k=\Omega(\lambda w)$ and $\lambda w = \Omega(1)$. Then there exists a family of $w$-sets $\calF \subseteq 2^{[n]}$ which is $k$-moonflower-free such that, any subset $I \subseteq [n]$ satisfying $|\calF_{\bar{I}}| \leq \exp (\lambda w)$ must have size $|I| \geq \Omega(\lambda^{-1}k)$. 
\end{lemma}

\subsection{Code sparsification}
One application of our extremal moonflower bounds and universe reduction theorem is \emph{code sparsification.} Code sparsification studies the following question: given an arbitrary code $\calC \subseteq \{0, 1\}^n$, can we restrict $\calC$ to only a few weighted coordinates while approximately preserving the weight of all codewords?

Formally speaking, let $\calC \subseteq \{0, 1\}^{n}$ be an arbitrary code. A \emph{weighted coordinate set} is a pair $(T, \alpha)$ where $T \subseteq [n]$ and $\alpha: T \to \rr_{\geq 0}$. It induces the estimator
\[
\widehat{\mathrm{wt}}_{T,\alpha}(x)\ :=\ \sum_{i\in T}\alpha(i)\,x_i,
\qquad x\in\{0,1\}^n.
\]
We say $(T,\alpha)$ \emph{$\varepsilon$-sparsifies} $\calC$ if
$\widehat{\mathrm{wt}}_{T,\alpha}(x)\in (1\pm\varepsilon)\mathrm{wt}(x)$ for all $x\in\calC$ where $\mathrm{wt}(x)=\sum_{i=1}^n x_i$. Such a pair $(T, \alpha)$ is called an \emph{$\varepsilon$-sparsifier} of $\calC$.

In the case where $\calC$ is a linear code, \cite{KPS24, KPS25} showed that in randomized $\poly (n, 1/\varepsilon)$ time, one can compute an $\varepsilon$-sparsifier with support size $|T| \leq \Tilde{O}(\dim (\calC) / \varepsilon^2)$. For general (not necessarily linear) codes, \cite{BG25} extended the results in \cite{KPS24, KPS25} and showed that any code $\calC \subseteq \{0, 1\}^n$ admits an $\varepsilon$-sparsifier with support size $|T| \leq O(\NRD (\calC) (\log n)^6 / \varepsilon^2)$ where $\NRD (\calC)$ denotes the \emph{non-redundancy} of $\calC$ defined as follows.

\begin{definition}[Non-redundancy \cite{BG25}]
    A subset $I \subseteq [n]$ is \emph{non-redundant} for a code $\calC \subseteq \{0, 1\}^n$ if for each $i \in I$, there exists $c \in \calC$ such that $c_i=1$ and $c_j=0$ for all $j \in I \setminus \{i\}$.
    We define the \emph{non-redundancy} of $\calC$, denoted by $\NRD(\calC)$, to be the size of the largest set that is non-redundant for $\calC$.
\end{definition}
If we view a code $\calC \subseteq \{0, 1\}^n$ as a $|\calC| \times n$ matrix, then it is clear from the definition that $\NRD(\calC)$ is the dimension of the largest permutation submatrix contained in $\calC$. When $\calC$ is a linear code, $\NRD(\calC)$ equals the dimension of $\calC$. More importantly for us, $\NRD$ is closely connected to moonflowers.

Given a set family $\calF$, let $\MF (\calF)$ denote the size of the largest moonflower contained in $\calF$. Denote $\calF_{\calC} := \{\supp (c): c \in \calC\} \subseteq 2^{[n]}$. Then we have $\MF (\calF_C) = \NRD (C)$. We introduce the new nomenclature instead of using $\NRD$ to draw a parallel with the bounds on classical sunflower lemma bounds, which inspired the extremal questions, e.g., \Cref{q:wsfext}, we study.

For a proof of the equivalence, see \Cref{lem:equivalence_between_NRD_and_WSF}. In particular, this allows us to apply \Cref{main_wsf_theorem} to bound the size of $\calC_{\leq d}:= \{c \in \calC: \mathrm{wt}(c) \leq d\}$  for any $d \in [n]$ as moonflower-freeness is preserved under projections (see \Cref{lem:WSF_free_preserved_under_projection}). Using this and a refined analysis of the techniques introduced in \cite{BG25}, we are able to obtain the following result on code sparsification.

\begin{theorem}[Improved code sparsification]\label{main:code_sparsification}
Let $\calC\subseteq\{0,1\}^n$ with $\mathrm{NRD}(\calC) = k$. 
Then for every $\varepsilon\in(0,1/4)$ there exists a weighted coordinate set $(T,\alpha)$
that $\varepsilon$-sparsifies $\calC$ and satisfies
\[
|T|
\ \le\
\frac{k\log n}{\varepsilon^2}\cdot \poly\!\left(\log(k/\varepsilon),\log\log n\right).
\]
\end{theorem}
Comparing with the result in \cite{BG25}, we bring the dependence on $n$ to be optimal up to $\poly(\log \log n)$ factors instead of being off by $\poly(\log n)$ factors. In addition, while we build on the work of \cite{BG25}, our arguments are simpler owing to the tight moonflower bound (which is also obtained by elementary means). 

As the following lemma shows, the $\log n$ dependence is necessary for $\varepsilon$-sparsifiers in general. A modification of an argument from \cite{BG25} shows the following:

\begin{lemma}[Lower bound on code sparsification]\label{main:sparsification_lower_bound}
Let $k \ge 1$ and $\varepsilon \in (0,1)$. Then, for all large enough $n$, there exists an explicit $\calC \subseteq \{0, 1\}^n$ with $\mathrm{NRD}(\calC)= k$ such that any $\varepsilon$-sparsifier $(T, \alpha)$ of $\calC$ must satisfy
    \[
    |T| =\Omega \left( \frac{k \log (n/k)}{\varepsilon} \right).
    \]
\end{lemma}
For $k = O(1)$, the above lemma implies $|T| = \Omega ((\log n) / \varepsilon)$, matching the upper bound in \Cref{main:code_sparsification} up to polynomial factors in $1/\varepsilon$. We prove this simple claim in \Cref{sec:sparsification_lower_bound}. It is an interesting question to see if the dependence on $\varepsilon$ in the lower bound can be improved, or if the upper bound dependency on $\varepsilon$ can be improved.

\paragraph{Chain-Length Complexity.} A quantity related to $\NRD$ called \emph{chain-length} complexity has also been studied in the context of sparsification in \cite{brakensiek2026multiplicative}. For a set $\mathcal{C} \subseteq \{0,1\}^n$, let $\CL(\calC)$ denote the length of the longest ascending chain in the union closure of the set-family $\mathcal{F}_{\calC} = \{supp(c): c \in \calC\}$. Since $\NRD(\calC) \leq \CL(\calC)$, the results of \cite{BG25} also imply sparsification in terms of $\CL(\calC)$. In \cite{brakensiek2026multiplicative}, the authors proved a sparsification upper bound of the form $O(\CL(\calC) \poly\log(\CL(\calC))/\epsilon^2)$. Note that there is no dependence on the universe-size $n$ which stands in contrast to our lower bound above. \cite{brakensiek2026multiplicative} also show a lower bound  of $\Omega(\CL(\calC)/\epsilon^2)$ for sparsifying some set systems via connections to data structure lower bounds for graph cuts in \cite{carlson2019optimal}.  

However, $\NRD$ can be much smaller than $\CL$ in general. If we take a matrix view of a code $\calC$, then $\CL (\calC)$ is the size of the largest upper triangular matrix while $\NRD(\calC)$ is the size of the largest identity matrix. For instance, take $\calC = \{(1,0,\cdots, 0), (1,1,0,\cdots,0),\ldots,(1,1,1,\cdots,1)\} \subseteq \{0,1\}^n$. Then $\NRD(\calC) = 1$, whereas $\CL(\calC) = n$. Here, our sparsification upper bound is nearly tight, whereas the bounds based on chain-length are notably worse. This makes our bounds incomparable to those of \cite{ brakensiek2026multiplicative} in general.

\subsection{Proof overview}
In this subsection, we present the overview of our proofs. 

\paragraph{Moonflowers.} Frankl \cite{Frankl1982ExtremalProblem} showed that if $\calA = \{A_1,...,A_m\}$ is a family of $r$-sets and $\calB = \{B_1,...,B_m\}$ is a family of $s$-sets such that (1) $A_i \cap B_i = \emptyset$ for $i = 1,2,...,m$ and (2) $A_i \cap B_j \neq \emptyset$ for $1 \leq i < j \leq m$. Then $m \leq \binom{r+s}{s}$.  

Given this result, we show that starting from a family of $w$-sets which is $k$-moonflower free, we can recursively define set families $\calA$ and $\calB$ using sets of minimal size and minimal intersections such that the resulting $\calA$ and $\calB$ satisfy the conditions of
\cite{Frankl1982ExtremalProblem}. Theorem \ref{main_wsf_theorem} then follows immediately. The resulting proof is a simple reduction to Frankl's classical statement. For more details, see \Cref{sec:moonflower}. 

\paragraph{Universe reduction.} We first illustrate how we prove the existence of a popular element in a large family of $w$-sets which is $k$-moonflower free. To this end, we crucially apply the entropy lemma \cite{Gil22,Saw23} used to obtain improved bounds on the union-closed set-family conjecture \cite{frankl1995extremal}. 

More concretely, let $\calD$ be a probability distribution over $2^{[n]}$. Gilmer's entropy lemma states the following: \emph{if there is no popular element under $D$, then taking unions of i.i.d copies of $A \sim \calD$ increases the entropy.} Take $\calD$ to be the uniform distribution over $\calF$. Using this, we obtain lower bounds on the entropy of a distribution over unions of $\calU_t (\calF)$, the set family consisting of unions of $t$ sets in $\calF$. Notice this implies a lower bound on the size of $\calU_t (\calF)$ as entropy is upper bounded by the logarithm of the domain size. On the other hand,
since moonflower-freeness is preserved under unions (\Cref{lem:moonflower_free_preserved_under_union}), we can upper bound $|\calU_t (\calF)|$ via the extremal moonflower bound (Theorem \ref{main_wsf_theorem}). Combining these two bounds yields Theorem \ref{thm:large_family_size_implies_existence_of_popular_element}.

Now we take a step further. Instead of the existence of a single popular element, we want to show the existence of a small set of elements that support most of the sets in $\calF$. To this end, consider the following definition. We say a set family $\calF \subseteq 2^{[n]}$ is \emph{$p$-covered} if there exists a distribution $Q$ on $[n]$ such that for every $S \in \calF$, $\Pr_{i \sim Q}[i \in S] \geq p$. Thus, if our family of interest $\calF$ is $p$-covered, then we can simply sample coordinates from $Q$. As long as we sample enough coordinates, then $p$-coveredness guarantees most sets in $\calF$ are contained in the sampled coordinates.

While the given set family $\calF$ may not necessarily be $p$-covered, by appealing to LP-duality, we show that if we allow removal of a few sets from $\calF$, something stronger holds true: for every $J \subseteq [n]$, if we define $\calF_J:= \{S \cap J: S \in \calF\}$, then there exists an exceptional set $\calS_J \subseteq \calF_J$ of size $|\calS_J| \leq M$ such that $\calF_J \backslash \calS_J$ is $p$-covered. We say $\calF$ is \emph{$(p, M)$-almost-covered} if it satisfies this condition.

The last piece of the puzzle is how we can bound the parameter $M$. Indeed, the validity of the entire argument rests on $M$ being reasonably small. To show this, we crucially rely on the extremal moonflower bound (\Cref{main_wsf_theorem}). Combining this with Gilmer's entropy lemma, we are able to obtain good control on $M$.

We note that the LP-duality and sampling argument were used in \cite{BG25} and a similar universe reduction result was proved. The key component was also Gilmer's entropy lemma. But to establish their result, they appealed to VC-dimension \cite{Sau72, She72}. Our proof is noticeably simpler, and we obtain nearly tight control of the size of the popular set owing in part to the optimal extremal bound.   

\paragraph{Code sparsification. } We now pivot our discussion to code sparsification. In what follows, we will start with the simplest sparsifier and gradually refine it to obtain our final optimal sparsifier. Given a code $\calC \subseteq \{0, 1\}^n$, the simplest sparsifier we can have is the following: \emph{include each coordinate independently with probability $1/2$ to form a set $T$ and assign each coordinate a weight of $2$.} This weighting scheme makes sure the weights are preserved in expectation.

However, this sparsifier has an immediate problem: the weights of the small-weight codewords are not necessarily preserved. For instance, suppose we have a codeword $c \in \calC$ with weight $\mathrm{wt}(c) = 1$. Then the only way for $c$'s weight to be preserved is to include the sole coordinate in $\supp (c)$ to the sparsifier and this happens only with probability $ 1/ 2$. This probability quickly goes down as we have more such low-weight codewords in $\calC$, deeming this simple sparsifier implausible.

To rectify this, we include the support of all low-weight codewords to the sparsifier. Specifically, set a threshold $w_{\min}$. Let $\calC_{\leq w_{\min }}$ denote the set of all codewords in $\calC$ with weight $\leq w_{\min}$. Include $\bigcup_{c \in \calC_{\leq w_{\min}}}\supp (c)$ into the sparsifier $T$ and assign each such coordinate a weight of 1. Then we add each coordinate $i \in [n] \backslash (\bigcup_{c \in \calC_{\leq w_{\min}}}\supp (c))$ to $T$ with probability $1/2$ independently and assign each such coordinate a weight of 2. Every codeword $c \in \calC_{w_{\min}}$ has its weight preserved exactly. For $c \in \calC$ such that $|c| \geq w_{\min}$, using a standard Chernoff bound, the probability that a codeword $c$'s weight is not preserved up to an additive $\varepsilon$ factor is at most $\exp (-O(\varepsilon^2 |c|) ) \leq \exp (-O(\varepsilon^2 w_{\min}) )$. Taking the union bound over $|\calC|$ elements, the probability of some codeword's weight not being preserved is at most $|\calC|  \exp (-O(\varepsilon^2 w_{\min}) ) $. Setting $w_{\min} = O(\log |\calC| / \varepsilon^2)$ yields a failure probability of $O(1)$.

Despite its viability, the sparsifier above suffers from two major inefficiencies: (i) Since we only set a single threshold $w_{\min}$, we have a single failure probability upper bound for all codewords with weight $\geq w_{\min}$. This can be extremely lossy for codewords with large weights. (ii) Because this simple approach does not have any control on the size $|\calC_{\geq w_{\min}}|$, the only thing we can do is to upper bound this quantity using $|\calC|$.

In \cite{BG25}, the authors propose the following improved sparsifier that overcomes some of the inefficiencies: for the low-weight regime ($|c| \leq w_{\min}$),  as before, the support of all low-weight codewords is added to the sparsifier. Now for weights $w \geq w_{\min}$, they split the weights into dyadic weight intervals of the form $[w, 2w]$. They then applied a version of \Cref{thm:universe_reduction} to identify an important set of coordinates for codewords over each such dyadic interval and add the identified coordinates to the sparsifier. Then finally, include each coordinate not yet in the sparsifier independently with probability $1/2$ and with weight $2$. This new approach has two important benefits: first, the dyadic intervals allow for better Chernoff bounds as we have a more careful treatment of the weights. Second, for the high-weight codewords, because of the small trace guarantee of the identified coordinates, the union bound is now taking over a much smaller set.

Now we discuss how we further improve upon the sparsifier in \cite{BG25}. Given a code $\calC \subseteq \{0, 1\}^n$, consider the set family $\calF_{\calC}:= \{\supp (c): c \in \calC\}$. We make the observation that $\NRD (\calC) = \MF (\calF_{\calC})$ where $\MF(\calF)$ denotes the size of the largest moonflower contained in set family $\calF$. For a proof of this, see \Cref{lem:equivalence_between_NRD_and_WSF}. We divide the weights into three regimes: low-weight ($|c| \leq w_{\low}$); medium-weight ($w_{\low}\leq |c| \leq w_{\high}$) and high-weight ($|c| \geq w_{\high}$) regimes. For the low-weight regime, we again include the support of all low-weight codewords into the sparsifier. For the medium-weight regime, we again use the dyadic intervals to identify coordinates to add to the sparsifier. However, instead of doing this all the way up to weight $n$, we stop at some weight $w \leq w_{\high}$ and include the other coordinates independently with probability $1/2$. The reason for this is the following: the trace bound we obtain from  \Cref{thm:universe_reduction} deteriorates as the weight increases. At that point, the bound $|\calF_{\leq w}| \leq \binom{k+w}{w}$ we obtain from the optimal moonflower bound in  \Cref{main_wsf_theorem} outperforms the bound in \Cref{thm:universe_reduction}.

One may argue that we are not including enough coordinates as we could from the dyadic intervals as in \cite{BG25}. However, this would not be an issue as we iterate the entire procedure above on the randomly sampled coordinates $S$. With high probability, the sampled coordinates $S$ in each iteration has its size halved. Hence, we need to iterate at most $\log n$ iterations and in fact, this is precisely where the $\log n$ factor occurs in our final sparsification result. Finally, we assign weights to coordinates based on the iteration at which they get included in the sparsifier. 

Using our  universe-reduction bound \Cref{thm:universe_reduction} and the moonflower bound \Cref{main_wsf_theorem} with appropriately chosen parameters, we are able to obtain a sparsifier with the guarantees in \Cref{main:code_sparsification}. For a more thorough discussion of our improved sparsification strategy, see \Cref{subsec:proof_sparsification}.

\paragraph{Remark.} 
An earlier version of the paper obtained similar results but with more complicated proofs. This version simplifies and streamlines many of the proofs, in particular the extremal bound for moonflowers and universe reduction.

\subsection{Organization}
In \Cref{sec:prelim}, we define the relevant terms and prove some elementary results used in the proof of moonflower bound and code sparsification. In \Cref{sec:moonflower}, we prove tight upper and lower bounds for the size of families of $w$-sets without moonflowers. In \Cref{sec:universe_reduction}, we prove our main universe reduction theorem. In \Cref{sec:sparsification_single_log_n}, we use this tight extremal bound and universe reduction theorem to prove our code sparsification theorem.

\subsection{Future directions}
\begin{enumerate}[(1)]
    \item  \textbf{From existence to algorithms.} Our results for both moonflowers and code sparsification are existential. A natural follow-up question is: can we make our existence results algorithmic?  
    \item \textbf{Better sparsification lower bound.} Our current code sparsification lower bound is near optimal in terms of its dependence on $\NRD (\calC)$ and $n$. However, it is not optimal in terms of the dependence on $\varepsilon$. We believe that the correct dependency on the error should be $1/\varepsilon^2$. In addition, it would be nice to have matching lower bound in the interesting special case where the code $\calC$ is roughly balanced.
    \item \textbf{More applications of moonflowers.} What other applications are there for moonflowers in combinatorics and theoretical computer science? Given the optimal bounds and the general structure of a moonflower, we believe moonflowers should have a broader range of applications. 
\end{enumerate}

\paragraph{AI acknowledgments.} While ChatGPT5.5-Pro was used in checking some of the calculations (especially for the somewhat tedious ones in code sparsification), all of the ideas in the paper are due to the authors. The only exception is \Cref{wsf_lower_bound} where ChatGPT improved our lower bound on $\MF(k,w)$ from ${k+w-2 \choose w}$ to ${k+w-1 \choose w}$ by modifying our initial construction.
ChatGPT5.5-Pro was also used in polishing parts of the writuep but the writing is mostly done by the authors. 

\paragraph{Acknowledgments. }The authors would like to thank anonymous FOCS reviewers for valuable comments.

%% file: preliminaries.tex
In this subsection, we state some elementary definitions and results that will be used throughout this paper. Throughout the paper, $\log (\cdot)$ denotes base-2 logarithms and $\ln (\cdot)$ denotes natural logarithms.
\subsection{Sets}

Let $[n]$ denote the finite set $[n]: = \{1,2,...,n\}$. Throughout this paper, unless otherwise stated, all sets are finite. We use normal letters like $S,T$ to denote sets and calligraphic letters like $\calF$ to denote family of sets. For $I \subseteq [n]$, write $\Bar{I}:= [n] \backslash I$. We use $2^{[n]}$ to denote the family of all subsets of $[n]$.

For $\calF \subseteq 2^{[n]}$, define $\supp(\calF) := \bigcup_{S \in \calF} S$. We say $S \subseteq [n]$ is a \emph{$w$-set} if $|S| \leq w$.


\begin{definition}[Moonflower]\label{definition:moonflower}
    A family of sets $S_1,....,S_k \subseteq [n]$ is a $k$-moonflower if there exists $I \subseteq [n]$ such that the sets $S_i \backslash I$ are all nonempty and pairwise disjoint. We refer to the $I$ of the smallest size satisfying this condition as the \emph{core} of the moonflower and $S_1,...,S_k$ as the \emph{petals}. A family $\calF \subseteq 2^{[n]}$ is \emph{$k$-moonflower-free} if it contains no $k$-moonflower.
\end{definition}

\noindent \textbf{Remark. }
It is worth noting that the core of a $k$-moonflower is unique: it is the set of all elements appearing in at least two of the sets $S_i$. 

Let $\calF$ be a family of sets. Define \emph{$\MF(\calF)$} to be the largest $k$ such that $\calF$ contains a $k$-moonflower.

\begin{definition}[Extremal function]\label{definition:extremal-moonflower}
    Let $\MF(k, w)$ denote the maximum size of a $k$-moonflower-free family of $w$-sets $\calF$ (over all universes $[n]$), i.e.,
    \[
    \MF(k, w) := \max_{\substack{\calF: \text{a family of }w-\text{sets} \\
    \text{s.t. }\MF(\calF) < k}} |\calF|
    \]
\end{definition}

\begin{definition}[Trace]\label{definition:trace}
    Let $\calF \subseteq 2^{[n]}$ be a set family and let $I \subseteq [n]$. The \emph{trace} of $\calF$ on $I$, denoted $\calF_I$ is
    \[
    \calF_I:= \{S \cap I: S \in \calF\}.
    \]
    In words, it is the family obtained by restricting every set $S \in \calF$ to the coordinates in $I$ with duplicate sets removed. We will be interested in particular in the trace after \emph{puncturing} the coordinates in $I$,
    \[
    \calF_{\Bar{I}}:= \{S \setminus I: S \in \calF\}.
    \]
\end{definition}

One important property of $k$-moonflower-freeness is that it is preserved under taking projections:
\begin{lemma}[Moonflower-freeness is preserved under projection]\label{lem:WSF_free_preserved_under_projection} If $\calF \subseteq 2^{[n]}$ is $k$-moonflower-free and $J \subseteq [n]$, then $\calF_J$ is $k$-moonflower-free.
\end{lemma}

\begin{proof}
    Suppose for contradiction that $\calF_J$ contains a $k$-moonflower: there exist $S_1',...,S_k' \in \calF_J$ and a core $I' \subseteq J$ such that the sets $S_i' \backslash I'$ are nonempty and pairwise disjoint. Choose $S_i \in \calF$ with $S_i \cap J = S_i'$. Let $I:= I' \cup ([n] \backslash J)$. Then $S_i \backslash I = (S_i \cap J) \backslash I' = S_i' \backslash I'$, hence nonempty and pairwise disjoint. So $S_1,...,S_k$ form a $k$-moonflower in $\calF$, contradiction.
\end{proof}

For a set family $\calF$, define its \emph{$t$-fold union}, denoted $\calU_t (\calF)$ as
\[
\calU_t (\calF) := \{S_1 \cup \cdots \cup S_t: S_1,\ldots,S_t \in \calF\}.
\]
It turns out taking $t$-fold unions also preserves moonflower-freeness.

\begin{lemma}[Moonflower-freeness is preserved under unions]\label{lem:moonflower_free_preserved_under_union}
    If $\calF$ is $k$-moonflower-free, then $\calU_t (\calF)$ is $k$-moonflower-free for all positive integers $t$.
\end{lemma}
\begin{proof}
    Suppose for contradiction that $\calU_t (\calF)$ contains a $k$-moonflower $A_1,\ldots,A_k$ where $A_i = \cup_{\ell=1}^t S_{i, \ell}$ with $S_{i, \ell} \in \calF$. For each $i \in [k]$, choose an element $x_i \in A_i$ that is not contained in $\cup_{j \ne i} A_j$ (such an element exists because $A_1,\ldots,A_k$ form a moonflower). In particular, $x_i$ must be contained in some $S_{i, \ell_i}$. Then $S_{1, \ell_1},\ldots,S_{k, \ell_k}$ form a $k$-moonflower in $\calF$, which is a contradiction. 
\end{proof}

\begin{lemma}[Support bound from moonflower-freeness]\label{lem:moonflower_free_family_have_small_support}Let $\calF \subseteq 2^{[n]}$ be a family of $w$-sets which is $k$-moonflower free. Then
\[
|\supp (\calF)| \leq (k-1)w.
\]
In particular, after deleting unused coordinates, we may assume $n \leq (k-1)w$. 
\end{lemma}

\begin{proof}
If $\mathcal F=\emptyset$ then $\supp(\mathcal F)=\emptyset$ and we are done.  Let $U:=\supp(\mathcal F)$.
Pick a subfamily $\mathcal F'\subseteq \mathcal F$ that is minimal (under inclusion) subject to
$\bigcup_{S\in\mathcal F'} S = U$.
Then for every $S\in\mathcal F'$ there exists an element
\[
x_S \in S\setminus \bigcup_{T\in\mathcal F',\,T\neq S} T,
\]
otherwise $S$ could be removed while still covering $U$, contradicting minimality.
In particular, the elements $\{x_S: S\in\mathcal F'\}$ are pairwise distinct.

Let $I := [n]\setminus \{x_S : S\in\mathcal F'\}$.  For each $S\in\mathcal F'$ we have
$S\setminus I = \{x_S\}$, so the sets in $\mathcal F'$ form a $|\mathcal F'|$-moonflower with core $I$.
Since $\mathcal F$ is $k$-moonflower-free, we get $|\mathcal F'|\le k-1$. Therefore
\[
|U|=\left|\bigcup_{S\in\mathcal F'} S\right|\le \sum_{S\in\mathcal F'} |S|
\le |\mathcal F'|\cdot w \le (k-1)w,
\]
as required.
\end{proof}

\subsection{Codes}

We define a Boolean code to be an arbitrary $\calC \subseteq \{0,1\}^n$. We say $\calC$ is \emph{non-trivial} if $\calC \neq \emptyset, \{0, 1\}^n$. For any $c \in \calC$, denote $\supp (c) \subseteq [n]$ to be the set of nonzero coordinates. Define $\supp (\calC) := \bigcup_{c \in C}\supp (c)$. For any codeword $c \in \calC$, we define its \emph{Hamming weight} as $| c |:= c_1 + ... + c_n$, i.e., the number of nonzero coordinates. For any $d \in [n]$, let $\calC_{\leq d}$ be the set of codewords in $\calC$ with Hamming weight at most $d$.

Given $S \subseteq [n]$ and $c \in \{0,1\}^n$, we define $c|_S \in \{0, 1\}^S$ to be the list $(c_i: i \in S)$. Likewise, we define \emph{punctured} code $\calC|_{S}:= \{c|_{S}: c \in \calC\} \subseteq \{0, 1\}^S$.

\begin{definition}[Non-redundancy \cite{BG25}]\label{definition:non-redundancy}
    A subset $I \subseteq [n]$ is \emph{non-redundant} for a code $\calC \subseteq \{0, 1\}^n$ if for each $i \in I$, there exists $c \in \calC$ such that $c_i=1$ and $c_j=0$ for all $j \in I \setminus \{i\}$.
    We define the \emph{non-redundancy} of $\calC$, denoted by $\NRD(\calC)$, to be the size of the largest non-redundant set that is non-redundant for $\calC$.
\end{definition}
If we view a code $\calC$ as a $|\calC| \times n$ matrix, then it is clear from the definition that $\NRD(\calC)$ is the dimension of the largest permutation submatrix contained in $\calC$. When $\calC$ is a linear code, $\NRD(\calC)$ equals the dimension of $\calC$. In addition, for any non-trivial $\calC$, we have $1 \leq \NRD(\calC) \leq n$. Let $\calF_{\calC}:= \{\supp(c): c\in \calC\}\subseteq 2^{[n]}$. We have the following simple lemma establishing a connection between $\NRD (\calC)$ and $\MF(\calF_{\calC})$.

\begin{lemma}[$\NRD$ equals $\MF$ of the support family]\label{lem:equivalence_between_NRD_and_WSF}
Let $\calC\subseteq\{0,1\}^n$ be a binary code. Then
\[
\MF(\mathcal F_{\calC})=\NRD(\calC).
\]
\end{lemma}

\begin{proof}
We prove the two inequalities.

\smallskip
\noindent\textbf{(1) $\MF(\mathcal F_{\calC})\ge \NRD(\calC)$.}
Let $I\subseteq[n]$ be non-redundant for $\calC$ with $|I|=\NRD(\calC)$. By Definition~2.11, for every $i\in I$
there exists a codeword $c^{(i)}\in \calC$ such that for all $i'\in I$,
\[
c^{(i)}_{i'}=1 \iff i'=i,
\]
equivalently, $\supp(c^{(i)})\cap I=\{i\}$.
Let $S_i:=\supp(c^{(i)})\in\mathcal F_{\calC}$ and set the core to be $J:=[n]\setminus I$.
Then for each $i\in I$,
\[
S_i\setminus J = S_i\cap I = \{i\},
\]
so the sets $\{S_i:i\in I\}$ form a $|I|$-moonflower (their petals are the singletons $\{i\}$, hence
nonempty and pairwise disjoint). Therefore $\MF(\mathcal F_{\calC})\ge |I|=\NRD(\calC)$.

\smallskip
\noindent\textbf{(2) $\NRD(\calC)\ge \MF(\mathcal F_{\calC})$.}
Let $k:=\MF(\mathcal F_{\calC})$, so $\mathcal F_{\calC}$ contains a $k$-moonflower
$S_1,\dots,S_k$ with some core $J\subseteq[n]$; write $P_t:=S_t\setminus J$ for the petals.
By definition, each $P_t$ is nonempty and the $P_t$'s are pairwise disjoint.
Choose an element $i_t\in P_t$ for each $t\in[k]$ and let $I:=\{i_1,\dots,i_k\}$.
For each $t$, let $c^{(t)}\in \calC$ be a codeword with $\supp(c^{(t)})=S_t$.
We claim that $\supp(c^{(t)})\cap I=\{i_t\}$.
Indeed, $i_t\in S_t$ by construction. If $s\neq t$, then $i_s\notin J$ (since $i_s\in P_s$) and
also $i_s\notin S_t\setminus J=P_t$ (since the petals are disjoint), hence $i_s\notin S_t$.
Thus $S_t\cap I=\{i_t\}$, meaning that on the coordinate set $I$, the codeword $c^{(t)}$ has a single
$1$ exactly at $i_t$.

Therefore, for every $i_t\in I$ there exists $c^{(t)}\in \calC$ whose restriction to $I$ equals the unit vector
at $i_t$. This is exactly the non-redundancy condition of Definition~2.11, so $I$ is non-redundant for $\calC$.
Hence $\NRD(\calC)\ge |I|=k=\MF(\mathcal F_{\calC})$.

Combining (1) and (2) yields $\MF(\mathcal F_{\calC})=\NRD(\calC)$.
\end{proof}

\subsection{Entropy amplification}

\begin{definition}[$p$-smooth distribution]
     A probability distribution $D$ supported on $2^{[n]}$ is \emph{$p$-smooth} if for every coordinate $i \in [n]$,
    \[
    \Pr_{S \sim D}[i \in S] \leq p.
    \]
\end{definition}

\begin{lemma}[Gilmer's entropy lemma {\cite[Corollary 4.9]{BG25}}]\label{lemma:entropy_lemma} Let $D$ be a $p$-smooth distribution over $2^{[n]}$. Let $t$ be a power of two such that $1 - (1-p)^{t/2} \leq \frac{3-\sqrt{5}}{2}$ and $A_1,\ldots,A_t$ be i.i.d samples from $\calD$. Then
\[
H \left(\bigvee_{i=1}^t A_i \right) \geq H(D) \cdot \frac{h(1 - (1-p)^t)}{h(p)}.
\]    
Here $h(\cdot)$ is the binary entropy function.
\end{lemma}
Note that if $p = \Theta (1 / t)$, then $h(1- (1-p)^t) / h(p) = \Theta( t / \log t )$.

\begin{lemma}[Entropy bound from size of $r$-fold unions, variant of {\cite[Corollary 4.13]{BG25}}]\label{lem:entropy_bound_from_size_of_r_fold_unions}
    Let $\calF \subseteq 2^{[n]}$ be any finite family. For $p \in (0, 1)$, let $D$ be a $p$-smooth distribution supported on $\calF$. If $r \geq 2/p$, then
    \[
    H(D) \leq O ( p \log (1/p) \cdot \log |\calU_r (\calF)|),
    \]
    where $H(\cdot)$ denotes the Shannon entropy.
\end{lemma}
For a proof of the statement, see  \cite[Corollary 4.13]{BG25}.

\subsection{Chernoff bound}
We need the following version of Chernoff bound for our sparsification result.
\begin{lemma}[Chernoff bound, convenient form]\label{lem:sparsif_chernoff_v2}
Let $t\ge 1$ and let $X\sim \mathrm{Bin}(t,1/2)$. Then for every $\Delta>0$,
\[
\Pr\left[\,|2X-t|>\Delta\,\right]\ \le\ 2\exp\!\left(-\frac{\Delta^2}{3t}\right).
\]
\end{lemma}

\begin{proof}
Standard Chernoff bound in additive form.
\end{proof}

\ignore{
\begin{lemma}[Lemma 3.3 of \cite{BG25}]\label{lem:bound_sparse_codewords_with_NRD}
    For all $C \subseteq \{0, 1\}^m$ and all $d \in \{0,1,...,m\}$, we have 
    \[
    |\supp (C_{\leq d})| \leq d \cdot \NRD (C).
    \]
\end{lemma}
\begin{proof}
    See proof of Lemma 3.3 in \cite{BG25}.
\end{proof}

We now define what it means to sparsify a code. Given a weight function $w: [n] \to \rr_{\geq 0}$, the $w$-weight of a codeword $c \in C \subseteq \{0,1\}^n$ is defined as
\[
\langle w, c \rangle := \sum_{i=1}^n w(i)c_i.
\]

\begin{definition}[$\epsilon$-sparsifier]
    For $\epsilon \in (0, 1)$, we say that $w: [n] \to \rr_{\geq 0}$ is a $\epsilon$-sparsifier of $C \subseteq \{0, 1\}^m$ if for all $c \in C$, we have
    \[
    (1-\epsilon) \Ham (c) \leq \langle w, c \rangle \leq (1+\epsilon) \Ham (c).
    \]
\end{definition}
Define the $\epsilon$-sparsifier of $C$, denoted by $\SPR(C, \epsilon)$ to be the minimum support size (i.e., number of nonzero coordinates) of any $\epsilon$-sparsifier of $C$.}
\ignore{
\subsection{Chernoff bounds}
We need the following versions of Chernoff bounds for our results. 

\begin{lemma}[Chernoff bound e.g. \cite{MU05}, \cite{BLM13}]\label{lem:chernoff_standard}
    Let $X_1,...,X_n$ be i.i.d. samples of the Bernoulli distribution with probability $p \in [0, 1]$. Then, for all $\delta > 0$,
    \begin{equation}\label{eqn:chernoff_smaller}
        \Pr \left[\sum_{i=1}^n X_i < (1-\delta)pn \right] \leq \exp (-\delta^2 np/2)
    \end{equation}
    \begin{equation}\label{eqn:chernoff_upper}
        \Pr \left[\sum_{i=1}^n X_i > (1+\delta)pn \right] \leq \exp (-\delta^2 np / (2+\delta)).
    \end{equation}
\end{lemma}

We obtain the following corollary.

\begin{corollary}\label{cor:chernoff_size_not_big}
    Let $S \subseteq [n]$ be a random subset of $[n]$ such that each element is included independently with probability at most $1/3$. Then,
    \[
    \Pr [|S| > n/2] \leq \exp (-n/30).
    \]
\end{corollary}
\begin{proof}
    Apply (\ref{eqn:chernoff_upper}) with $p = 1/3$ and $\delta = 1/2$.
\end{proof}}

\ignore{
Lastly, we need the following lemma which shows random sampling preserves the weight of a codeword with high probability.

\begin{lemma}\label{lem:chernoff_random_samp_weight}
    Let $S \cup T$ be a partition of $[n]$. For $p \in (0, 1]$, let $S_p$ be a random subset of $S$ where each element of $S$ is included independently with probability $p$. For any codeword $c \in \{0, 1\}^n$ and any $\epsilon \in (0,1 )$, we have that
    \[
    \Pr_{S_p} \left[\frac{1}{p}\text{Ham}(c|_{S_p}) + \text{Ham}(c|_T) \notin [1-\epsilon, 1+\epsilon] \cdot \text{Ham}(c) \right] < 2 \exp (-\epsilon^2 \text{Ham}(c) p / 3).
    \]
\end{lemma}
\begin{proof}
    Let $w_S = \text{Ham}(c|_S)$ and $w_T= \text{Ham}(c|_T)$, so $\text{Ham}(c) = w_S + w_T$. Notice $\text{Ham}(w_{S_p})$ is the summation of $n$ independent Bernoulli random variables. Applying Lemma \ref{lem:chernoff_standard} with $n: = w_S$ and $\delta := \epsilon \cdot \frac{w_S + w_T}{w_S}$, we have
    \[
    \Pr_{S_p} \left[\text{Ham}(c|_{S_p}) \notin [1-\delta, 1+\delta] \cdot pn \right] < 2\exp (-\epsilon^2 \text{Ham}(c) p / 3).
    \]
    Lastly, notice
    \[
    \begin{split}
        \text{Ham}(c_{S_p}) \notin [1-\delta, 1+\delta]pn &\Leftrightarrow \frac{1}{p} \text{Ham}(c|_{S_p}) + \text{Ham}(c|_T) \notin [1-\delta, 1+\delta] \cdot w_S + w_T \\
        &\Leftrightarrow  \frac{1}{p} \text{Ham}(c|_{S_p}) + \text{Ham}(c|_T) \notin [1-\epsilon, 1+\epsilon] \cdot \text{Ham}(c),
    \end{split}
    \]
    as desired.
\end{proof}}

%% file: moonflower_bound.tex
In this section, we prove our extremal moonflower bound. 
\begin{theorem}[\Cref{main_wsf_theorem}, restated]\label{thm:moonflower_bnd}
    Let $\calF$ be a family of $w$-sets which is $k$-moonflower-free. Then $|\calF| \leq \binom{k+w-1}{w}$.
\end{theorem}

Towards the proof of this theorem, we need the following result of Frankl \cite{Frankl1982ExtremalProblem}.
\begin{theorem}[\cite{Frankl1982ExtremalProblem}]\label{thm:frankl}
    Let $\calA = \{A_1,\ldots,A_m\}$ be a family of $r$-sets and $\calB = \{B_1,\ldots,B_m\}$ be a family of $s$-sets such that (1) $A_i \cap B_i = \emptyset$ for $i = 1,2,\ldots,m$ and (2)  $A_i \cap B_j \neq \emptyset$ for $1 \leq i < j \leq m$. Then 
    \[
    m \leq \binom{r+s}{s}.
    \]
\end{theorem}

 We include a proof of \Cref{thm:frankl} in \Cref{sec:appendix} for completeness.

\begin{proof}[Proof of \Cref{thm:moonflower_bnd}] Let $\calF $ be a family of $w$-sets which is $k$-moonflower free, and let $m=|\calF|$. Let $U:= \bigcup_{S \in \calF} S$ denote the universe. We construct families $\calA$ and $\calB$ as follows. Let $\calF_1 = \calF$,  $\calA=\calB = \emptyset$. For $i = 1,2,\ldots,m$ do:
\begin{itemize}
    \item Take $B_i \in \calF_i$ of minimal size. In particular, $B_i$ does not contain any other set in $\calF_i$. 
    \item Let $A_i \subseteq U \setminus B_i$ (not necessarily a set in $\calF_i$) be a set of minimal size intersecting all sets in $\calF_i \setminus \{B_i\}$. Such a set exists because $B_i$ does not contain any other set in $\calF_i$.
    \item Add $A_i$ to $\calA$ and $B_i$ to $\calB$.
    \item Set $\calF_{i+1} := \calF_i \setminus \{B_i\}$ and repeat.
\end{itemize}
Denote the resulting set families as $\calA = \{A_1,\ldots,A_m\}$ and $\calB = \{B_1,\ldots,B_m\}$. The resulting families $\calA$ and $\calB$ satisfy the following properties:
\begin{enumerate}[(i)]
    \item $A_i \cap B_i = \emptyset$ for all $i = 1,2,\ldots,m$.
    \item $A_i \cap B_j \neq \emptyset$ for $1 \leq i < j \leq m$.
    \item $|A_i| \leq k-1$ for all $i = 1,2,\ldots,m$.
\end{enumerate}
Notice that (i) holds as $A_i$ is contained in $U \setminus B_i$, (ii) holds because $A_i$ intersects all $S \in \calF_i \setminus \{B_i\}$ by definition, which implies that $A_i \cap B_j \neq \emptyset$ for all $j > i$. It remains to show (iii).

Since $A_i$ is minimal, removing any element from $A_i$ would miss some set; 
that is, for each $s \in A_i$, there exists a set $T_s \in \calF_i$ such that $A_i \cap T_s = \{s\}$. This implies that $\{T_s\}_{s \in A_i}$ form a $|A_i|$-moonflower which is a contradiction if $|A_i| \geq k$. 

Consequently, we have obtain a $(k-1)$-set family $\calA$ and a $w$-set family $\calB$ satisfying the conditions of \Cref{thm:frankl}.  Applying \Cref{thm:frankl}, we obtain 
\[
|\calF| \leq \binom{k+w-1}{w}.
\]
\end{proof}

We next prove that the bound is tight.

\begin{lemma}[\Cref{wsf_lower_bound}, restated]\label{lem:WSF_lower_bnd_restated}
    There exists a family of $w$-sets $\calF$ such that $|\calF| = \binom{k+w-1}{w}$ and $\calF$ is $k$-moonflower-free. 
\end{lemma}

\begin{proof}
    For $s=0,\ldots,w$ define the following family of sets
    \[
    \calF_s := \left\{S \subseteq [k+s-2]: |S| = s \right\}.
    \]
    Let $\calF = \cup_{s=0}^w \calF_s$. We claim that $\calF$ satisfies the requirement of the lemma.
    First, note that $\calF$ is a family of $w$-sets of size
    \[
    |\calF| = \sum_{s=0}^w {k+s-2 \choose s} = {k+w-1 \choose w}.
    \]    
    We next show that $\calF$ does not contain a $k$-moonflower. Assume towards a contradiction that it contains a $k$-moonflower $S_1,\ldots,S_k$. For each $i=1,\ldots,k$ let $x_i$ be an element of $S_i$ which is not in any other $S_j$. Assume without loss of generality that $x_1$ is maximal among $x_1,\ldots,x_k$, and let $s=|S_1|$. Note that $S_1 \subset [k+s-2]$ is a set of size $s$, and $x_2,\ldots,x_k \in [k+s-2]$ are $k-1$ elements distinct from it. This leads to a contradiction.
\end{proof}

%% file: universe_reduction.tex
In this section, we use the extremal moonflower bound (\Cref{thm:moonflower_bnd}) to obtain two structural results on moonflower-free sets. The main technical ingredient of the proofs is \emph{entropy amplification} \cite{Gil22, Saw23}. As a warm-up, in \Cref{subsec:large_family_implies_popular}, using entropy amplification, we show that any large $k$-moonflower-free set family must contain a popular element. Then, for the rest of this section, we work our way towards proving our main universe reduction theorem (\Cref{thm:universe_reduction}).

Before we start with the proof, we need to make a
definition that will play an important role in the proof of \Cref{thm:universe_reduction}. 

\begin{definition}[$p$-covered family]
    A family $\calG \subseteq 2^{[n]}$ is \emph{$p$-covered} if there exists a distribution $Q$ on $[n]$ such that for every $S \in \calG$,
    \[
    \Pr_{i \sim Q} [i \in S] \geq p.
    \]
    Equivalently, $\sum_{i \in S} Q(i) \geq p$ for all $S \in \calG$.
\end{definition}

It is worth noting that the distribution $Q$ in a $p$-covered family is the LP dual of a $p$-smooth distribution $D$. Moreover, if we know a $p$-cover $Q$ for a set family $\calF$ exists, we can simply sample from $Q$. The main steps of the proof of \Cref{thm:universe_reduction} can then be described as follows:
\begin{enumerate}
    \item We first show that any $p$-smooth distribution $D$ supported on a $k$-moonflower free family $\calF$ has entropy at most $O( \log \binom{k+w/p}{k} \cdot p  \log(1/p))$. This step critically uses the entropy-amplification property of taking unions proved by \cite{Gil22, Saw23}. See \Cref{subsec:universe_reduction}. 
    \item We then use LP-duality to show the following. Suppose $\calF$ has the property that for any  $J \subseteq [n]$, any $p$-smooth distribution on $\calF_J$ has entropy at most $H$. Then, there exists a small set of bad elements $\calS \subset \calF$, $|\calS| \leq 2^H$ such that $\calF \setminus \calS$ is $p$-covered. See \Cref{subsec:entropy_to_cover}. 
    \item Given this, we show that if a family of sets $\calF$ has the following property: for any $J \subseteq [n]$, there exists $\calS_J \subset \calF_J$ of size $|\calS_J| \leq 2^H$ such that $\calF_J\setminus \calS_J$ is $p$-covered. Then, we can find a small set $I \subseteq [n]$ supporting most of the sets in $\calF$. The size of $I$ depends on the quantity $H$. \Cref{thm:universe_reduction} then follows by bounding the quantity $H$ using extremal moonflower bound (\Cref{thm:moonflower_bnd}). See \Cref{subsec:universe_reduction}
    \item Finally, in \Cref{subsec:universe_reduction_LB}, we complement our universe reduction result with a lower bound.
\end{enumerate}

\textbf{Remark. }The idea of using LP-duality and sampling can be found in \cite{BG25}. With our new extremal moonflower bound (\Cref{thm:moonflower_bnd}), we are able to significantly simplify the proof in \cite{BG25}. 

\subsection{Large family size implies existence of popular elements}\label{subsec:large_family_implies_popular}

\begin{theorem}[Theorem \ref{thm:large_family_size_implies_existence_of_popular_element}, restated]\label{lem:large_family_size_implies_existence_of_popular_element}
    Let $\calF \subseteq 2^{[n]}$ be a family of $w$-sets which is $k$-moonflower-free. If $|\calF| > {\binom{k+tw}{tw}}^{\Omega((\log t)/t)}$  for some integer $t \geq 2$, then there exists an element $i \in [n]$ such that $i$ is contained in more than $1/2t$ fraction of the sets in $\calF$. 
\end{theorem}

The following corollary simplifies the binomial coefficient involved by breaking it down into three cases depending on the relative magnitude of $k$ and $tw$.
\begin{corollary}
     Let $\calF \subseteq 2^{[n]}$ be a family of $w$-sets which is $k$-moonflower-free. 
     Assume for some integer $t \ge 2$ that all elements $i \in [n]$ are contained in less than a $1/2t$ fraction of sets in $\calF$. Then 
     \[
     |\calF| \leq
     \begin{cases}
         \left( \frac{k}{tw}\right)^{O(w \log t)} &\text{when }k \gg tw \\
         t^{O(w)} &\text{when } k \asymp tw \\
         \left( \frac{tw}{k}\right)^{O(k\log t / t)} &\text{when } k \ll tw.
     \end{cases}
     \]
\end{corollary}
\begin{proof}
    Use the standard approximation for binomial coefficients.
\end{proof}
Towards a proof of \Cref{lem:large_family_size_implies_existence_of_popular_element}, we first use \Cref{lemma:entropy_lemma} to prove the following corollary that bounds the size of $t$-fold unions provided the uniform distribution over $\calF$ is $p$-smooth. Recall that $\calU_t(\calF):= \{S_1 \cup ... \cup S_t: S_1,...,S_t \in \calF\}$ is the $t$-fold union of $\calF$.

\begin{lemma}[Smooth uniform distribution implies large union size]\label{lem:bound_t_union_size}
    Let $\calF$ be a finite set family. Suppose the uniform distribution over $\calF$ is $(1/2t)$-smooth, then for some constant $c_1$, we have
    \[
    |\calU_t (\calF)| \geq |\calF|^{c_1t / \log t}.
    \]
\end{lemma}
\begin{proof}
    For a fixed $t$ that is a power of 2, let $p = c / t$ for some constant\footnote{need this to ensure the condition $1 - (1-p)^{t/2} \leq \frac{3-\sqrt{5}}{2}$.} $c \le 0.7$. Let $D$ be the uniform distribution on $\calF$, and let $A_1,\ldots,A_t \sim D$ be uniformly chosen. Apply \Cref{lemma:entropy_lemma} to get
    \[
    \log |\calU_t (\calF)| \geq H \left(\bigvee_{i=1}^t A_i \right) \geq H(D) \cdot \frac{h(1 - (1-p)^t)}{h(p)} \ge  \log |\calF| \cdot \frac{c_1 t}{\log t}
    \]
    for some constant $c_1$. Here we used the fact that the distribution $\vee_{i=1}^t A_i$ is supported on 
    $\calU_t(\cal F)$, and so its entropy is bounded by $\log |\calU_t(\calF)|$, and that $D$ is uniform on $\calF$ and so its entropy equals $\log |\calF|$. Raising both sides to the power of 2 and simplifying, we obtain the desired bound.
\end{proof}

\begin{lemma}\label{lem:smooth_uniform_implies_small_family_size}
    Let $\calF$ be a family of $w$-sets which is $k$-moonflower-free. Suppose the uniform distribution over $\calF$ is $(1/2t)$-smooth, then for some absolute constant $C>0$ we have 
    \[
    |\calF| \le {\binom{k+tw}{tw}}^{C (\log t)/t}.
    \]
\end{lemma}
\begin{proof}
    On one hand, by \Cref{lem:moonflower_free_preserved_under_union}, we know $\calU_t(\calF)$ is a family of $tw$-sets and $k$-moonflower free. By \Cref{thm:moonflower_bnd}, we know 
    \[
    |\calU_t (\calF)| \leq \binom{k+tw-1}{tw}.
    \] On the other hand, by \Cref{lem:bound_t_union_size}, we have $|\calU_t (\calF)| \geq |\calF|^{c_1t / \log t}$. We thus find
    \[
    \binom{k+tw-1}{tw}  \geq |\calF|^{c_1t / \log t}.
    \]
    The desired statement then follows by rearranging.
\end{proof}

\begin{proof}
    (of \Cref{lem:large_family_size_implies_existence_of_popular_element}) By \Cref{lem:smooth_uniform_implies_small_family_size}, we know if $|\calF| \ge {\binom{k+tw}{tw}}^{C (\log t)/t}$ for some large enough constant $C$, then the uniform distribution over $\calF$ cannot be $(1/2t)$-smooth. This implies there must exist some element $i \in [n]$ that belongs to at least a $(1/2t)$-fraction of the sets in $\calF$. 
\end{proof}

\subsection{Entropy bound}\label{subsec:universe_reduction}

In this subsection, show the following entropy bound on $p$-smooth distributions over $k$-moonflower-free $w$-set families.

\begin{theorem}[Entropy bound on $p$-smooth distributions]\label{thm:ent_bnd_p_smooth}
    Let $\calF$ be a $k$-moonflower-free family of $w$-sets. Let $D$ be a $p$-smooth distribution supported on $\calF$ for $0 < p \leq 1/2$. Then 
    \[
    H(D) \leq O\left( p \log (1/p) \cdot \log \binom{k+w \lceil 2/p \rceil}{k}\right).
    \]
\end{theorem}
\begin{proof}
    Let $r = \lceil 2/p \rceil$.
    Consider the $r$-fold union family $\calU_{r}(\calF)$. By \Cref{lem:entropy_bound_from_size_of_r_fold_unions}, we know
    \[
    H(D) \leq  O ( p\log(1/p) \cdot \log |\calU_r (\calF)|).
    \]
    It remains to bound $|\calU_r (\calF)|$. Notice that every set in $\calU_r(\calF)$ has size at most $wr$. Moreover, by \Cref{lem:moonflower_free_preserved_under_union}, we know $\calU_r (\calF)$ is $k$-moonflower-free. We can then apply the extremal moonflower bound (\Cref{thm:moonflower_bnd}) and bound
    \[
    |\calU_r (\calF)| \leq \binom{k+wr}{k}.
    \]
    The desired statement follows. 
\end{proof}

\subsection{From entropy to coverability}\label{subsec:entropy_to_cover}
In this section, we show how to obtain a $p$-cover of $\calF$. Namely, we are going to prove the following theorem. 
\begin{theorem}[Coverability from entropy-bounded smoothness, adapted from {\cite[Lemma~4.15]{BG25}}]
\label{thm:coverability_from_smoothness}
    Let $\calF \subseteq 2^{[n]}$ and $p \in (0, 1)$. Assume every $p$-smooth distribution supported on $\calF$ has entropy $\leq H$. Then there exists $\calS \subseteq \calF$ with $|\calS| \leq 2^H$ such that $\calF \backslash \calS$ is $p$-covered.
\end{theorem}

While the statement is similar to \cite[Lemma 4.15]{BG25}, we provide a simplified proof. First, we establish the duality of coveredness and smoothness.

\begin{proposition}[Cover vs. smooth duality]\label{prop:minimax_smooth_cover}
    Let $\calG \subseteq 2^{[n]}$ and $p \in (0, 1)$. Define
    \[
    \Phi (\calG) := \max_{Q \text{ distribution on }[n]} \min_{T \in \calG} \sum_{i \in T}Q(i).
    \]
    Then 
    \[
    \Phi (\calG) = \min_{D \text{ distribution on }\calG} \max_{i \in [n]} \Pr_{T \in D}[i \in T].
    \]
    In particular,
    \begin{enumerate}
        \item $\calG$ is $p$-covered if and only if $\Phi (\calG) \geq p$;
        \item If $\calG$ is not $p$-covered, then there exists a $p$-smooth distribution supported on $\calG$.
    \end{enumerate}
\end{proposition}

\begin{proof}
    Consider the following finite zero-sum game: the row player chooses $T \in \calG$ and the column player chooses $i \in [n]$, with payoff matrix $A(T, i):= \mathbf{1}[i \in T]$ to the column player. A mixed strategy for the column player is a distribution $Q$ on $[n]$, and its expected payoff against a pure row choice $T$ is $\sum_{i \in T}Q(i)$. The row player then respond with the worst expected payoff $\min_{T \in \calG} \sum_{i \in T}Q(i)$. Maximizing over all mixed strategy $Q$ gives $\Phi (\calG)$.

    A mixed strategy for the row player is a distribution $D$ on $\calG$, and the expected payoff against a pure column choice $i$ is $\Pr_{T \sim D}[i \in T]$. The column player best-responds with $\max_i \Pr_{T \sim D}[i \in T]$. The row player then minimizes this quantity.

    By von Neumann's minimax theorem \cite{Neu28},
    \[
    \max_{Q} \min_{T \in \calG} \sum_{i \in T}Q(i) = \min_{D} \max_{i \in [n]} \Pr_{T \in D}[i \in T],
    \]
    as desired. 
\end{proof}

Recall our goal is to obtain a $p$-cover of $\calF$. If our $\calF$ is already $p$-covered, then we are done. Otherwise, by \Cref{prop:minimax_smooth_cover} there has to exists a $p$-smooth distribution on $\calF$. Given such a $p$-smooth distribution, our next Lemma shows how we can obtain a $p$-covered subfamily from $ \calF$ by removing only a few sets provided every $p$-smooth distribution on $\calF$ does not have too much entropy.

\begin{proof}[Proof of  \Cref{thm:coverability_from_smoothness}]
    If $\calF$ is already $p$-covered, take $S = \emptyset$ and we are done. Otherwise by \Cref{prop:minimax_smooth_cover} (2), there exists at least one $p$-smooth distribution over $\calF$. Let 
    \[
    \calP := \{\nu \text{ distribution on }\calF: \nu \text{ is }p \text{-smooth}\}.
    \]
    Then $\calP$ is nonempty, compact and convex. Define $\| \nu \|_{\infty}:= \max_{T \in \calF} \nu (T)$. This is valid since $\calF$ is finite. Let $\tau^* := \min_{\nu \in \calP} \| \nu \|_{\infty}$. By compactness, the minimum is attained. Fix an optimizer $\nu^* \in \calP$
    and define
    \[
    \calS := \{T \in \calF: \nu^* (T) = \tau^*\}.
    \]
    We claim that $\calS$ is a family of sets satisfying the desired conditions. To see this, we need to show (1) $\calF \backslash \calS$ is $p$-covered and (2) $|\calS| \leq 2^H$. \\
    \\
    \textbf{Step 1: $\calF \backslash \calS$ is $p$-covered.} If $\calF \backslash \calS = \emptyset$, this is trivially true. Otherwise, suppose for contradiction that $\calF \backslash \calS$ is not $p$-covered. Then \Cref{prop:minimax_smooth_cover}(2) yields a $p$-smooth distribution $\mu$ supported on $\calF \backslash \calS$. For $\epsilon \in (0, 1)$, define $\nu_{\epsilon}:= (1-\epsilon)\nu^* + \epsilon \mu$. Since this is a convex combination of $p$-smooth distributions over $\calF$, we have $\nu_{\epsilon} \in \calP$. Now,
    \begin{itemize}
        \item For any $S \in \calS$, $\mu (S) = 0$ as $\mu$ is supported on $\calF \backslash \calS$. Therefore, $\nu_{\epsilon}(S) = (1-\epsilon)\tau^* < \tau^*$.
        \item For any $T \in \calF \backslash \calS$, we have $\nu^* (T) < \tau^*$ by the definition of $S$. Since $\calF$ is finite, the minimum gap $\delta:= \min_{T \in \calF \backslash \calS} (\tau^* - \nu^* (T))$ is strictly positive. As $\mu (T) \leq 1$ for all $T$, we find
        \[
        \nu_{\epsilon}(T) = (1-\epsilon)\nu^* (T) + \epsilon \mu (T) = \nu^* (T) + \epsilon (\mu (T) - \nu^*(T)) \leq \nu^* (T) + \epsilon.
        \]
        Therefore if we choose $\epsilon \in (0, \delta)$, we get  
        \[
        \nu_{\epsilon}(T) \leq \nu^* (T) + \epsilon <\nu^* (T) + \delta  \leq \tau^*
        \]
        for all $T \in \calF \backslash \calS$, contradicting the optimality of $\nu^*$. Hence, $\calF \backslash \calS$ must be $p$-covered.
    \end{itemize}
    \noindent \textbf{Step 2: $|\calS| \leq 2^H$.} Since $\nu^*$ assigns mass $\tau^*$ to each $S \in \calS$, we have $|\calS| \tau^* \leq 1$ which implies $|S| \leq 1 / \tau^*$. On the other hand, 
    \[
    H(\nu^*) = \sum_{T \in \calF} \nu^* (T) \log \frac{1}{\nu^* (T)} \geq  \sum_{T \in \calF} \nu^* (T) \log \frac{1}{\|\nu^* \|_{\infty}} = \log \frac{1}{\tau^*}.
    \]
    Since by assumption $H(\nu^*) \leq H$, this implies $\tau^* \geq 2^{-H}$. As a result, $|S| \leq 2^H$. 
    
\end{proof}

\subsection{Universe reduction}\label{subsec:universe_reduction}
We are now ready to prove \Cref{thm:universe_reduction}.  We start by making the following definition:

\begin{definition}[$(p, M)$-almost-covered]\label{definition:goodness} Let $\calF \subseteq 2^{[n]}$. We say $\calF$ is \emph{$(p, M)$-almost-covered} if for every $J \subseteq [n]$, there exists an exceptional set $\calS_J \subseteq \calF_J$ of size $|\calS_J| \leq M$ such that $\calF_J \backslash \calS_J$ is $p$-covered.

\end{definition}

\ignore{We need the following lemma in the proof of Theorem \ref{theorem:one-step-universe-reduction-good-sets}.
\begin{lemma}[existence of popular element]\label{lem:existence_of_popular_elements}
    Let $\calF \subseteq 2^{[n]}$ be a $(p, M)$-almost-covered family of sets. Given $I \subseteq [n]$, let $\calF':= \{S \in \calF: S \not\subset I\}$ be the sets not fully-covered by $I$. If $|\calF'| \geq 2M$, then there is some $i \in [n] \backslash I$ which is contained in at least $p/2$ fraction of sets in $\calF'$.
\end{lemma}

\begin{proof}
    Notice $\calF' = \calF_{\Bar{I}}$ \shachar{why? I don't think this is true}. Since $\calF$ is $(p, M)$-almost-covered, we know there exists an exceptional set $\calS_{\Bar{I}} \subseteq \calF_{\Bar{I}}$ with $|\calS_{\Bar{I}}| \leq M$ such that $\calF_{\Bar{I}} \backslash \calS_{\Bar{I}}$ is $p$-covered. Therefore, by the definition of $p$-coveredness,  there exists a distribution $Q$ on $\Bar{I}$ such that for all $S \in \calF_{\Bar{I}} \backslash \calS_{\Bar{I}}$, $\sum_{i \in S}Q(i) \geq p$. In particular,
    \[
    \sum_{S \in \calF_{\Bar{I}} \backslash \calS_{\Bar{I}}}\sum_{i \in S}Q(i) \geq p \cdot |\calF_{\Bar{I}} \backslash \calS_{\Bar{I}}| \geq p \cdot \frac{|\calF'|}{2}.
    \]
    On the other hand, switching the order of summation, we find \shachar{why use sum over 1s instead of set size?}
    \[
     p \cdot \frac{|\calF'|}{2} \leq \sum_{S \in \calF_{\Bar{I}} \backslash \calS_{\Bar{I}}}\sum_{i \in S}Q(i) = \sum_{i \in \Bar{I}}Q(i) \left(\sum_{S \in \calF_{\Bar{I}} \backslash \calS_{\Bar{I}}: i \in S}1 \right)
    \]
    Since $\sum_{i \in |I|}Q(i) = 1$, there must exist an $i \in \Bar{I}$ such that 
    \[
    \sum_{S \in \calF_{\Bar{I}} \backslash \calS_{\Bar{I}}: i \in S}1 \geq p \cdot \frac{|\calF'|}{2},
    \]
    as desired.
\end{proof}
}

We show that if a set family $\calF$ is almost-covered, then we can identify a small set of coordinates $I$ that covers most of the sets in $\calF$.

\begin{lemma}\label{lemma:one-step-universe-reduction-good-sets}
    Let $\calF \subseteq 2^{[n]}$ be a family of $w$-sets. Fix $p \in (0, 1), M \geq 1$ and $\delta \in (0, 1]$. If $\calF$ is $(p, M)$-almost-covered, then there exists an $I \subseteq [n]$ with $|I| \leq t $ such that 
    \[
    |\mathcal F_{\overline I}|
    \le
    \delta|\mathcal F|+tM.
    \]
    where $t:= \lceil \frac{2}{p}(w \ln 2) + \ln (1/\delta) \rceil$.
\end{lemma}
In words, the lemma is saying that there exists a small universe of size $t$ on which most sets in $\calF$ are preserved.

\begin{proof}
    Without loss of generality, we can assume $\calF$ is non-empty. Otherwise, we can take $I = \emptyset$ and the conclusion follows trivially. Since $\calF$ is $(p, M)$-almost-covered by assumption, we can fix a choice of witness $(\calS_J, Q_J)$ for each $J \subseteq [n]$ as in \Cref{definition:goodness}. That is, $\calS_J \subset \calF_J$ of size $|\calS_J|\le M$, and $Q_J$ is a distribution over $J$ that $p$-covers $\calF_J \setminus \calS_J$. We now describe a process that iteratively selects a set of important elements that cover most sets in $\calF$. Let $t_{\text{end}}$ be some stopping condition to be determined later. Then:
    \begin{itemize}
        \item Initialize $I_0 := \emptyset, J_0 = [n], \calS_{\text{removed}}^{(0)} = \emptyset$ and $\calF_0 = \calF$.
        \item For $j = 0,1,...,t_{\text{end}}-1$:
        \begin{itemize}
            \item Let $(\calS_{J_j}, Q_{J_j})$ denote the corresponding witness for $J_j$.
            \item Remove exceptions: $\calF_{j}^{-} := \calF_j \backslash \calS_{J_j}.$ 
            \item Sample $i_{j+1} \sim Q_{J_j}$ and add it: $I_{j+1} = I_j \cup \{i_{j+1}\}$. 
            \item Update $J_{j+1}:= [n] \backslash I_{j+1}$ and $\calS_{\text{removed}}^{(j+1)} = \calS_{\text{removed}}^{(j)} \cup \calS_{J_j}$. 
            \item Set $\calF_{j+1} = (\calF_{j}^-)_{J_{j+1}}
            $, and remove any empty sets or duplicate sets if some exist.
        \end{itemize}
    \end{itemize} 
    Now define the potential function\footnote{Since $A \in \calF_j$, we have $|A \cap J_j| = |A|$. Here we write $|A \cap J_j|$ to emphasize it is the size of an intersection.}: 
    \[
    \Phi_j := \sum_{A \in \calF_j}2^{|A \cap J_j|}.
    \]
    It is easy to see that $\Phi_j$ is decreasing in $j$ and $\Phi_j \geq 2 |\calF_j|$. Moreover, since $\calF$ is a family of $w$-sets, we know $\Phi_0 \leq |\calF| \cdot 2^w$. Let $t_{\text{end}}$ be the smallest $j$ such that $\E [\Phi_j] \leq \delta |\calF|$. We claim $t_{\text{end}}:= \min \{n, \lceil \frac{2}{p}(w \ln 2) + \ln (1/\delta) \rceil\}$.

    To this end, we first observe that the process can continue at iteration $j$ as long as $|\calF_j| \geq M$. Lets assume this is always the case before we finish (we will justify this assumption later).   Now condition on the choice of $i_1,\ldots,i_j$ so that $I_j, J_j, \calF_j$ are fixed. Let $A \in \calF_{j}^{-}$. Then we have $A \notin \calS_{J_j}$, which implies  $\Pr_{i_{j+1} \sim Q_{J_j}}[i_{j+1} \in A] \geq p$. Moreover, we know
    \[
    |A \cap J_{j+1}| = |A \cap J_{j}| - \mathbf{1}[i_{j+1} \in A].
    \]
    Therefore,
    \[
    \E_{i_{j+1} \sim Q_{J_j}} \left[2^{|A \cap J_{j+1}|}\Bigm| i_1,...,i_j \right] = \E[2^{|A \cap J_{j}| - \mathbf{1}[i_{j+1} \in A]}] = 2^{|A \cap J_{j}|} \cdot \E [2^{- \mathbf{1}[i_{j+1} \in A]}] \leq 2^{|A \cap J_{j}|} \cdot \left(1 - \frac{p}{2} \right).
    \]
    By linearity of expectation, we thus find
    \[
    \begin{split}
        \E_{i_{j+1 \sim Q_{J_j}}}\left[ \Phi_{j+1} \Bigm| i_1,...,i_j\right] &= \sum_{B \in \calF_{j+1}} \E_{i_{j+1 \sim Q_{J_j}}}\left[ 2^{|B \cap J_{j+1}|} \Bigm| i_1,...,i_j\right] \\
        &\leq  \sum_{A\in \calF_{j}^-} \E_{i_{j+1 \sim Q_{J_j}}}\left[ 2^{|A \cap J_{j+1}|} \Bigm| i_1,...,i_j\right] \\
        &\leq \sum_{A \in \calF_{j}^-}2^{|A \cap J_{j}|} \cdot \left(1 - \frac{p}{2} \right) \\
        &\leq \sum_{A \in \calF_j}2^{|A \cap J_j|} \cdot \left(1 - \frac{p}{2} \right) = \left(1 - \frac{p}{2} \right) \Phi_j,
    \end{split}
    \]
    where the first inequality holds since for every $B \in \calF_{j+1}$, there exists some $A \in \calF_{j}^-$ such that $B \subseteq A$. The last inequality holds since $\calF_{j}^-$ is a subset of $\calF_j$. Iterating and using $1-x \leq e^{-x}$, we obtain
    \[
    \E_{i_1,...,i_t} [\Phi_t] \leq \left(1 - \frac{p}{2} \right)^t \Phi_0 \leq \exp (-pt/2) \cdot (|\calF| \cdot 2^w).
    \]
    Setting $\delta = 2^w \cdot \exp (-pt_{\text{end}}/2)$ and solving for $t_{\text{end}}$, we get $t_{\text{end}} = \lceil \frac{2}{p}(w \ln 2) + \ln (1/\delta) \rceil$. Now by the first moment method, we know there exists a sequence of $I = \{i_1,...,i_{t_{\text{end}}}\}$ such that $\sum_{A \in \calF_{t_{\text{end}}}}2^{|A|} \leq \delta |\calF|$. Therefore, the number of sets in $\calF$ that are not covered by $I$ is
    \[
    |\calF_J| \leq |\calF_{t_{\text{end}}}| + |\calS_{\text{removed}}^{t_{\text{end}}}| \leq \delta |F| + t_{\text{end}}M.
    \]
    The desired result then follows since $I$ can have size at most $n$. Lastly, if during any iteration before termination we encounter the case where $|\calF_j| \leq M$ which clearly satisfies the statement. Hence, our assumption $|\calF_j| \geq M$ for all iterations is a valid one.    
\end{proof}

We have thus seen $(p, M)$-almost-coveredness is a desirable property as it allows us to reduce the size of the universe without affecting the size of the set family too much. It remains to understand how we can get such a nice property. It turns out we can obtain almost-coveredness from moonflower-freeness.
\begin{lemma}[Almost-coveredness from moonflower-freeness]\label{lem:WSF_free_sets_are_good} Let $\calF \subseteq 2^{[n]}$ be a $k$-moonflower-free family of $w$-sets. Then $\calF$ is $(p, 2^h)$-almost-covered where 
\[
h = h(w,k,p):=O\left(p \log(1/p) \log \binom{k+w \lceil 2/p \rceil}{k} \right).
\]
\end{lemma}

\begin{proof}
    Since $\calF$ is $k$-moonflower-free, by \Cref{lem:WSF_free_preserved_under_projection}, each projection $\calF_J$ is also $k$-moonflower-free. By \Cref{thm:ent_bnd_p_smooth}, every $p$-smooth distribution $D$ supported on $\calF_J$ must satisfy
    \[
     H(D) \leq O \left( p \log(1/p)  \log \binom{k+w \lceil 2/p \rceil}{k} \right).
    \]
  Now by  \Cref{thm:coverability_from_smoothness}, for each $J$ there exists an exceptional set $\calS_J \subseteq \calF_J$ with $|\calS_J| \leq 2^h$ such that $\calF_J \backslash \calS_J$ is $p$-covered. In particular, this implies $\calF$ is $(p, 2^h)$-almost-covered.
    
\end{proof}

We are now ready to prove \Cref{thm:universe_reduction}, restated below for convenience.

\begin{theorem}[Universe reduction, \Cref{thm:universe_reduction} restated]
    Let $\calF \subseteq 2^{[n]}$ be a $k$-moonflower-free family of $w$-sets. Then for any $\lambda \in (0, 1]$, there exists $I \subseteq [n]$ of size $|I| \leq O( \lambda^{-1}k L^3)$ satisfying $|\calF_{\bar{I}}| \leq \exp (\lambda w)$ where $L:= \log(kw / \lambda)$. 
\end{theorem}

\begin{proof}
    Assume without loss of generality that $|\calF| > \exp (\lambda w)$. Otherwise, take $I = \emptyset$ and we are done. Let $p:= c \lambda w / (k L^2)$ where $c$ is some sufficiently small constant, and; we have\footnote{This follows as $p \le \frac{c \log |\calF|}{k \log^2 (kw/\lambda)} \le \frac{c \log \binom{k+w}{k}}{k \log^2 (kw/\lambda)}\leq \frac{c k \log (e(1 + w/k))}{k \log^2 (kw/\lambda)} < 1/2$ for $c$ small enough.} $p < 1/2$. Let $r = \lceil 2/p \rceil$.
    By \Cref{lem:WSF_free_sets_are_good}, we know $\calF$ is $(p, 2^h)$-almost-covered with $h:= O( p \log(1/p) \cdot \log \binom{k+wr}{k})$. For our choice of $p$, we have $\log (1/p) = O(L)$  and
    \[
    \log \binom{k + wr}{k} \leq k \log \left(\frac{e(k+wr)}{k}\right) \leq O(k L).
    \]
    Consequently, $h \leq O(c\lambda w) \leq \lambda w / 2$ for $c$ a small enough constant. Applying \Cref{lemma:one-step-universe-reduction-good-sets} with $p:= c \lambda w / (k L^2)$, $M = 2^h$ and $\delta = 1/|\calF|$, we get that there exists $I \subseteq [n]$ such that $|I| \leq t:= \lceil \frac{2}{p}\log(|\calF|2^w)\rceil + \log |\calF|$ and $|\calF_{\bar{I}}| \leq t \cdot 2^h$. By the extremal moonflower bound (\Cref{thm:moonflower_bnd}), we get
    \[
    t \leq O \left(\frac{\log |\calF| + w}{p} + \log |\calF|\right) \leq O \left(\frac{\log \binom{k+w}{w} + w}{p}\right) \leq O\left(\frac{w L}{p}\right) \leq O\left( \frac{k L^3}{\lambda} \right),
    \]
    where the last inequality holds by definition of $p$.
    If $t < 2^{\lambda w / 2}$, then $|\calF_{\bar{I}}| \leq t \cdot 2^{\lambda w / 2} \leq \exp (\lambda w)$ and we are done. 

    Otherwise, we have $\lambda w \leq 2 \log t$. In this case we take $I:= \supp (\calF)$. Then $\calF_{\bar{I}} = \{\emptyset\}$. Moreover, by \Cref{lem:moonflower_free_family_have_small_support}, we know
    \[
        |I| \leq kw \leq \frac{2k \log t}{\lambda} \leq O\left(\frac{kL}{\lambda}\right) 
    \]
    which is even better than the claimed bound.
\end{proof}

\subsection{Universe reduction lower bounds}\label{subsec:universe_reduction_LB}
In this subsection, we prove lower bounds on the size of punctured sets needed for universe reduction. We first justify why the interesting regime of parameters is $k \gg \lambda w \gg 1$:
\begin{enumerate}
    \item If $k=O(\lambda w)$ then by the extremal moonflower bound, $|\calF| \le {k+w-1 \choose k} \le \exp(O(\lambda w \log(1/\lambda)))$ which is already close enough to the desired trace bound $\exp(\lambda w)$ by slightly adjusting $\lambda$.
    \item If $\lambda=O(1/w)$ then the desired trace bound $\exp(\lambda w)$ is $O(1)$; and the support of $\calF$ is $|\supp(\calF)|\le (k-1)w \le O(k/\lambda)$. So we can simply prune all elements in the support.
\end{enumerate}

\begin{lemma}[\Cref{lem:universe_reduction_lower_bnd}, restated]\label{lem:universe_reduction_lower_bnd_restated}
        Let $k, w \geq 2$ be positive integers and $\lambda \in (0,1]$. Assume that $k=\Omega(\lambda w)$ and $\lambda w = \Omega(1)$.  
    Then there exists a family of $w$-sets $\calF \subseteq 2^{[n]}$ which is $k$-moonflower-free such that any subset $I \subseteq [n]$ satisfying $|\calF_{\bar{I}}| \leq \exp (\lambda w)$ must have size $|I|=\Omega(\lambda^{-1} k)$.
\end{lemma}

\begin{proof}
    Define $r := 1/\lambda$ and $m:= \lambda w$, and assume without loss of generality that $r,m$ are integers, by changing $\lambda, w$ by constant factors if needed. Let $p:= k+m-2$ and $\calG$ be the family of all subsets of $[p]$ of size $m$. We saw in \Cref{lem:WSF_lower_bnd_restated} that $\calG$ is $k$-moonflower free.
    Next, consider the set family $\calF$ defined as $\calF:= \{A \times [r]: A \in \calG\} \subseteq 2^{[p] \times [r]}$, and note that all sets in $\calF$ have size $mr \le w$. We claim that $\calF$ is also $k$-moonflower free. Indeed, if $S_1,\ldots,S_k$ form a $k$-moonflower in $\calF$, then by construction $S_i = A_i \times [r]$ and $A_1,\ldots,A_k \in \calG$ also form a $k$-moonflower, a contradiction.

    Next, consider an arbitrary set $I \subseteq [p] \times [r]$. For each $a \in [r]$, define $I_a:= \{x \in [p]: (x, a) \in I\}$. Notice that $|\calF_{\bar{I}}| \geq |\calG_{\bar{I_a}}|$ for all $a \in [r]$. Hence, if $|\calF_{\bar{I}}| \leq \exp (m)$, then we necessarily have $|\calG_{\bar{I_a}}| \leq \exp (m)$ for all $a \in [r]$. 

    On the other hand, let $A \subseteq [p]$. Then by the definition of $\calG$, $\calG_{\bar{A}}$ contains all subsets of size $m$ from the remaining $p - |A|$ coordinates. Hence, $|\calG_{\bar{A}}| \geq \binom{p-|A|}{m} \geq \left(\frac{p-|A|}{m}\right)^m \geq \exp (m)$ provided $p-|A| \geq em$. Therefore, if  $|\calG_{\bar{A}}| \leq \exp (m)$, we must have $p- |A| < em$. Equivalently, $|A| > k- (e-1)m - 2$.

    Applying this with $A = I_{a}$, we know $|I_a| \geq k - (e-1)m - 2$ for all $a \in [r]$. As a result,
    \[
    |I| = \sum_{a \in [r]}|I_a| \geq r(k - (e-1)m - 2) \geq \frac{k}{\lambda} - O(w),
    \]
    where the last inequality follows by our choice of $r$ and $m$. As we assume $k = \Omega(\lambda w)$, we conclude that $|I|=\Omega(k/\lambda)$. 
\end{proof}

\ignore{
     Consider the set family $\calF:= \{B_i\}_{i=1}^{k-1}$ such that $B_i \cap B_j = \emptyset$ for all distinct $i, j \in [k-1]$. Then clearly, $\calF$ is $k$-moonflower-free. Notice that given any $I \subseteq [n]$, 
    \[
        |\calF_{\bar{I}}| = \sum_{i=1}^{k-1}\mathbf{1}[B_i \text{ not contained in }I] \geq (k-1) - \frac{|I|}{w}.
    \]
    If $|\calF_{\bar{I}}| \leq \exp (\lambda w)$, then we must have $|I| \geq w(k-1-\exp(\lambda w))$.
}

\ignore{
\subsection{Universe reduction}\label{subsec:universe_reduction}

In this subsection, we prove our main universe reduction theorem. 

\begin{theorem}[\Cref{thm:main_sampling_thm}, restated]\label{lem:main_sampling_thm}
    Let $\calF \subseteq 2^{[n]}$ be a family of $w$-sets which is $k$-moonflower-free. For every $t \geq 2$ and every $\eta \in (0, 1)$, there exists a set $I \subseteq [n]$ with 
    \[
    |I| \leq O \left(\frac{t}{\eta} \log \binom{k+w-1}{w} \right)
    \]
    such that
    \[
    |\calF_{\bar{I}}| \leq O \left(\frac{1}{\eta} \cdot \exp \left(O(\eta w) + O \left(\frac{\log t}{t} \log \binom{k+tw-1}{tw} \right) \right) \right).
    \]
    In particular, if $k = \Theta (w)$, we can identify a set $I$ of size $I = O(k)$ such that $|\calF_{\bar{I}}| \leq 2^{k/10}$.
\end{theorem}

For the remainder of this subsection, fix some $t \geq 2$ and $\eta \in (0, 1)$. To prove the theorem, we need the following weighted popularity lemma.

\begin{lemma}[Weighted popularity from bounded weight range]\label{lem:weighted_popularity_bounded_weight_range} Let $\calF \subseteq 2^{[n]}$ be a $k$-moonflower-free family of sets with $r - \Delta \leq |S| \leq r$ for all $S \in \calF$. Let $\mu_{\calF}$ be the exponentially weighted distribution over $\calF$ such that $\mu_{\calF} (S) \propto 2^{|S|}$ for all $S \in \calF$. If $|\calF| > 2^{\Delta} N_r (t)$, then there exists an element $i \in [n]$ such that
\[
\Pr_{S \sim \mu_{\calF}}[i \in S] \geq \frac{1}{t}.
\]
Here $N_r(t):= \exp (C \cdot \frac{\log t}{t} \log \binom{k+tr-1}{tr})$.    
\end{lemma}

\begin{proof}
    Suppose for contradiction that $\mu_{\calF}$ is $(1/t)$-smooth. Let $A_1,\ldots,A_t$ be i.i.d samples from $\mu_{\calF}$. By \Cref{lemma:entropy_lemma}, we have
    \[
     H \left(\bigvee_{i=1}^t A_i \right) \geq c_1 H(\mu_{\calF}) \cdot \frac{t}{\log t}
    \]for some constant $c_1$. On the other hand, since $A_1 \cup \cdots \cup A_t \in \calU_t (\calF)$, we know $H(A_1 \cup \cdots \cup A_t) \leq \log |\calU_t (\calF)|$. By \Cref{lem:moonflower_free_preserved_under_union}, $\calU_t (\calF)$ is also $k$-moonflower-free. We can then apply the extremal moonflower bound (\Cref{thm:moonflower_bnd}). This gives $|\calU_t (\calF)| \leq \binom{k+tr-1}{tr}$. Combining the above, we thus have
    \[
    H(\mu_{\calF}) \leq \frac{1}{c_1} \frac{\log t}{t} \log \binom{k+tr-1}{tr}.
    \]
    It remains to lower bound $H(\mu_{\calF})$ in terms of $|\calF|$. To do this, notice that 
    \[
    \mu_{\calF}(S) = \frac{2^{|S|}}{ \sum_{T \in \calF}2^{|T|}} \leq \frac{2^r}{|\calF|2^{r-\Delta}} = \frac{2^{\Delta}}{|\calF|}
    \]
    where the inequality holds since $|S| \in [r-\Delta, r]$ for all sets $S \in \calF$. As a result,
    \[
    H(\mu_{\calF}) = \E_{S \sim \mu_{\calF}}\left[\log \frac{1}{\mu_{\calF}(S)}\right] \geq \log |\calF| - \Delta.
    \]
    This implies 
    \[
    \log |\calF| \leq \Delta + \frac{1}{c_1} \cdot \frac{\log t}{t} \cdot \log \binom{k+tr-1}{tr`}
    \]
    which is a contradiction for $C = 1/c_1$. We can then conclude that $\calF$ is not $(1/t)$-smooth, as desired.
\end{proof}

Using this, we prove the following one-step weight dropping lemma.

\begin{lemma}[One-step weight dropping]\label{lem:one-step-weight-dropping}
    Let $\calF$ be a $k$-moonflower-free family of $r$-sets. Fix some $\Delta \in [1, r]$. Then there exists a set $J \subseteq [n]$ with $|J| \leq O(t \log |\calF|)$ such that 
    \begin{equation}\label{eqn:one-step-weight-dropping}
        |\{S \setminus J: S \in \calF, |S \setminus J| \geq r - \Delta\}| \leq 2^{\Delta} N_r (t). 
    \end{equation}
\end{lemma}
In words, after puncturing the coordinates in $J$, only few large sets remain.

\begin{proof}
    We maintain a punctured coordinate set $J$, initially set to be $\emptyset$ and a current trace family $\calF_{\Bar{J}}:= \{S \setminus J: S \in \calF\}$. Let $\calG_J: = \{T \in \calF_{\Bar{J}}: |T| > r - \Delta\}$. We repeat the following procedure as long as $|\calG_J| > 2^{\Delta} N_r (t)$: let $i$ be the coordinate such that $\Pr_{S \sim \mu_{\calG_J}}[i \in S] \geq 1/t$ where $\mu_{\calG_J}(S) \propto 2^{|S|}$. Such a coordinate must exist by \Cref{lem:weighted_popularity_bounded_weight_range}. We add $i$ to $J$, puncture $i$ and repeat after moving duplicate traces.

    Notice that \eqref{eqn:one-step-weight-dropping} holds by the definition of the process. It remains to bound the number of puncturing steps. To do this, consider the potential function
    \[
    \Phi (J):= \sum_{S \in \calG_J}2^{|S|}.
    \]
    Let $J' = J \cup \{i\}$, where $i$ is the popular coordinate chosen. Since $i$ is $(1/t)$-popular under $\mu_{\calG_J}$, we have
    \[
    \Phi (J') = \sum_{S \in \calG_{J'}}2^{|S|} \leq \Phi(J) - \frac{1}{2}\sum_{S \in \calG_J: i \in S}2^{|S|} \leq \left(1 - \frac{1}{2t}\right) \Phi (J)
    \]
    where the last inequality holds since $\Pr_{S \sim \mu_{\calG_J}}[i \in S] \geq 1/t$. Notice that initially, we have $\Phi (\emptyset) \leq 2^{r}|\calF|$. On the other hand,  during every step of the process, we have $\calG_J > 2^{\Delta}N_r(t)$ and every set in $\calG_J$ has size $> r - \Delta$. This implies $\Phi (J) > 2^{\Delta} N_r (t) \cdot 2^{r- \Delta} = 2^r N_r (t)$. Combining the above, we find that the total number of puncturing steps is at most 
    \[
    O \left(t \log \frac{2^r |\calF|}{2^r N_r (t)} \right) = O \left(t \log \frac{|\calF|}{N_r(t)} \right) \leq O(t \log |\calF|),
    \]
    as desired.
\end{proof}

Now we are ready to formally prove \Cref{thm:main_sampling_thm}. 

\begin{proof}
    (of \Cref{thm:main_sampling_thm}) We iteratively apply \Cref{lem:one-step-weight-dropping}. To this end, let $\eta \in (0,1)$ and set $\Delta:= \lceil \eta w \rceil$. Let $L:= \lceil w / \Delta \rceil \leq O(1/ \eta)$. We process the weight levels from top to bottom. At stage $q = 0,1,\ldots,L-1$, let $r_q:= w - q \Delta$. 

    Apply \Cref{lem:one-step-weight-dropping} to the current unprocessed trace family, consisting only of sets of size at most $r_q$. Denote this family as $\calG_q$. The lemma punctures $O(t \log |\calG_q|)$ coordinates and leaves behind at most $2^{\Delta}N_{r_q}(t)$ traces of size larger than $r_q - \Delta$. We set aside these remaining large traces and continue with the traces of size at most $r_q - \Delta$. We claim that at the end of this process, the punctured set satisfies the desired guarantees.

    To see this, we first bound the size of the punctured set $I$. For each weight level $q$, at most $O(t \log |\calG_q|)$ coordinates are included in $I$. Since every intermediate trace family is a trace of $\calF$, it remains $k$-moonflower-free by \Cref{lem:WSF_free_preserved_under_projection}. We can thus apply \Cref{thm:moonflower_bnd} and obtain
    \[
    |\calG_q| \leq \binom{k+w-1}{w}.
    \]
    Since there are $L = O(1/ \eta)$ states, the total number of punctured coordinates is
    \[
    |I| \leq O(1/\eta) \cdot O\left( t \log \binom{k+w-1}{w}\right) = O\left(\frac{t}{\eta} \log \binom{k+w-1}{w} \right).
    \]
    We now bound the trace size\footnote{It is worth noting that the large traces set aside in the earlier stages may change in later stages as more coordinates are punctured. But its size would not increase, justifying the trace bound}. Since $N_{r_q}(t) \leq N_w (t)$ for all $q$, the total number of leftover traces is at most 
    \[
    1 + \sum_{q=0}^{L-1}2^{\Delta} N_{r_q}(t) \leq 1 + L2^{\Delta}N_w(t),
    \]
    we include the $1$ here to account for potential empty sets. Substituting in $L = \lceil 1 / \eta \rceil$ and $\Delta = \lceil \eta w \rceil$ and simplifying, we have
    \[
    |\calF_{\bar{I}}| \leq O \left(\frac{1}{\eta} \cdot \exp \left(O(\eta w) + O \left(\frac{\log t}{t} \log \binom{k+tw-1}{tw} \right) \right) \right).
    \]
    Finally, if $k = c w$ for some constant $c$, we have $\log \binom{k+tw-1}{tw} \leq O(w \log t)$. Choosing $t$ to be a sufficiently large constant and $\eta$ sufficiently small as a constant depending only on $c$ and $t$, we have $|\calF_{\bar{I}}| \leq 2^{k/10}$ and $|I| = O(k)$.
\end{proof}
}

%% file: newsparsification.tex


In this section, we refine the sparsification analysis from~\cite{BG25}.
The only dependence on the ambient blocklength $n$ is the \emph{single} $\log n$ factor
coming from the recursion depth.
All additional losses are $\poly(\log(k/\varepsilon),\log\log n)$.

Throughout, let $\calC\subseteq\{0,1\}^n$ be a code with $\mathrm{NRD}(\calC) \le k-1$.
Equivalently, the support family $\Supp(\calC):=\{\supp(x):x\in\calC\}$ is
$k$-moonflower-free by \Cref{lem:equivalence_between_NRD_and_WSF}.

A \emph{weighted coordinate set} is a pair $(T,\alpha)$ where $T\subseteq[n]$ and
$\alpha:T\to\mathbb{R}_{\ge 0}$. It induces the estimator
\[
\widehat{\mathrm{wt}}_{T,\alpha}(x)\ :=\ \sum_{i\in T}\alpha(i)\,x_i,
\qquad x\in\{0,1\}^n.
\]
We say $(T,\alpha)$ \emph{$\varepsilon$-sparsifies} $\calC$ if
$\widehat{\mathrm{wt}}_{T,\alpha}(x)\in (1\pm\varepsilon)\mathrm{wt}(x)$ for all $x\in\calC$ where $\mathrm{wt}(x)=\sum_{i=1}^n x_i$. Such a pair $(T, \alpha)$ is called an \emph{$\varepsilon$-sparsifier} of $\calC$. Our main sparsification result is as follows.
\begin{theorem}[Improved sparsification with a single $\log n$ factor]
\label{thm:sparsif_single_log_n_v2}
Let $\calC\subseteq\{0,1\}^n$ satisfy $\mathrm{NRD}(\calC)\le k-1$.
Then for every $\varepsilon\in(0,1/4)$ there exists a weighted coordinate set $(T,\alpha)$
that $\varepsilon$-sparsifies $\calC$ and satisfies
\[
|T|
\ \le\
\frac{k\log n}{\varepsilon^2}\cdot \poly\!\left(\log(k/\varepsilon),\log\log n\right).
\]
\end{theorem}

\subsection{Proof of  \Cref{thm:sparsif_single_log_n_v2}}\label{subsec:proof_sparsification}

The proof of \Cref{thm:sparsif_single_log_n_v2} is a recursive ``puncture--then--halve'' process
run for $R=\lceil\log_2 n\rceil$ rounds. The main calculations to keep in mind are:

\begin{itemize}
\item 
Each round adds a puncturing set $I_r$ of size at most
$\frac{k}{\varepsilon^2}\cdot \poly(\log(k/\varepsilon),\log\log n)$.
Since there are $R=\Theta(\log n)$ rounds, the final puncturing set has size 
$|T| \le \frac{k \log n}{\varepsilon^2}\cdot \poly(\log(k/\varepsilon),\log\log n)$. The remaining coordinates will be small due to halving in each round. 
\item The puncturing is carried-out carefully in that for each weight range $w$, we choose a specific puncturing set $I_{r,w}$. For round $r$, let $\mathcal{F}^{(r)}_{(w,2w]}$ denote codewords whose weight in round $r$ is between $w, 2w$ (in the non-punctured portion). 

We ensure that the number of traces of this family outside $I_{r,w}$ is roughly $\exp(\eta_0^2 w)$ where $\eta_0 \approx \varepsilon$. This makes the union bound combined with Chernoff viable for the codewords whose weights are in $[w,2w]$: with good probability, randomly sampling half the remaining coordinates preserves the weight of all such codewords within a factor of $(1\pm \eta_0)$. The one caveat is that we need the failure probability to be at most $1/\poly(\log n)$ so we have a minimum weight threshold $w_{min}$ above which we use this argument. 

\item The above argument is still problematic as each round incurs a multiplicative error of $(1+\eta_0)$ and we have $O(\log n)$ rounds. The simple fix of taking $\eta_0 \ll \varepsilon/(\log n)$ would in turn blow up the support size to be $\poly(\log n)$ which we want to avoid. 

The key point is that the error in the Chernoff bound combined with union bound argument can be better for large weight codewords. We set a weight-threshold $w_{\mathrm{high}} \approx O((k \log (k/\epsilon) + \log \log n))/\epsilon^2$ and do the following. Pick a dyadic weight $w = 2^i w_{\mathrm{low}} > w_{\mathrm{high}}$. For such codewords, we do no puncturing, and instead union-bound directly over such codewords using our optimal moonflower bound. Essentially, the number of codewords whose weight is between $(w,2w]$ is at most $(Cw/k)^k$. This allows to say that with high probability, randomly picking half the remaining codewords preserves the weights up to error roughly 
$$\eta(w) = O(\sqrt{ (k \log (w/k) + \log \log n)/w}.$$

Note that the error decreases as $w$ increases. 

\item Error accumulation: Finally, we do an error analysis for each codeword conditioned on the union bounds succeeding across all rounds. Fix a codeword $x \in \mathcal{C}$ and let $w_r$ denote its weight among non-punctured coordinates in round $r$. Then, the multiplicative error, $1 \pm \epsilon_r$, incurred by the sampling process satisfies
$$
\epsilon_r := \begin{cases} 0 & w_r \leq w_{\low},\\
\eta_0 & w_{\low} < w_r \leq w_{\high}\\
\eta(w_r) & w_r > w_{\high}
\end{cases}
$$

In each round, the residual weight of a codeword drops by a constant factor as long as its weight is above $w_{\mathrm{low}}$. Thus, it crosses the medium region $(w_{\mathrm{low}}, w_{\mathrm{high}}]$ in only $O(\log w_{\mathrm{high}})$ rounds and the total contribution to the multiplicative error through these rounds is
$O(\eta_0\log w_{\mathrm{high}})$, and we set $\eta_0=\Theta(\varepsilon/\log w_{\mathrm{high}})$.

There can be several rounds where the codeword would be large weight (i.e, weight above $w_{\mathrm{high}}$). However, the error for these rounds is better, and we get a convergent geometric series for the errors here because of the improved bound on $\eta(w)$ above. 

\ignore{
\item \textbf{Medium vs.\ large weights (the key stratification).}
For \emph{medium weights} $w\le w_{\high}$, we enforce small trace complexity via puncturing and then union-bound
over traces. For \emph{large weights} $w>w_{\high}$, we do \emph{no puncturing} and instead union-bound directly
over codewords using the extremal bound
$|\calF^{(r)}_{(w,2w]}|\le (C w/k)^k$, while allowing a larger per-round error
$\eta(w)=\Theta(\sqrt{(k\log(e w/k)+\log R)/w})$ so that the Chernoff exponent
$\eta(w)^2 w$ matches $\log |\calF^{(r)}_{(w,2w]}|+\log R$.

\item \textbf{Error accumulation (why we do not set $\eta_0=\varepsilon/R$).}
For a fixed codeword, the residual weight drops by a constant factor each round
($W_{r+1}\approx W_r/2$), so it crosses the medium region $(w_{\min},w_{\high}]$ in only
$O(\log w_{\high})$ rounds. Thus the total medium contribution to the multiplicative error is
$O(\eta_0\log w_{\high})$, and we set $\eta_0=\Theta(\varepsilon/\log w_{\high})$.
Large-weight rounds are even fewer and contribute a convergent geometric series.
Formally this is captured by the recursion $\xi_r=\eta_r+(1+\eta_r)\xi_{r+1}$,
which implies $1+\xi_0=\prod_r(1+\eta_r)\le \exp(\sum_r \eta_r)$.}
\end{itemize}

\begin{proof}(of \Cref{thm:sparsif_single_log_n_v2})

\paragraph{Parameter choices.}
Fix $R:=\lceil \log_2 n\rceil$. In what follows, let $C$ be a large constant. Define the transition threshold \ignore{\shachar{we have 3 unspecified constants $C$ below, do we assume they are all the same? if not then lets call them $C_1,C_2,C_3$; this then needs to be propagated into the proof} \yimeng{I think we can assume they are the same. As long as it's large enough, the proof should go through.}\shachar{at least it should be clear if $C$ should be large enough or small enough. EG if we incrase $C$ does it keep everything consistent?}}
\[
w_{\mathrm{high}}\ :=\ \left\lceil
C\cdot \frac{k\log(k/\varepsilon)+\log R}{\varepsilon^2}
\right\rceil.
\]
Define the medium-weight per-round error
\[
\eta_0\ :=\ \frac{\varepsilon}{100\log(2w_{\mathrm{high}})}.
\]
Define the low-weight cutoff
\[
w_{\mathrm{low}}\ :=\ \left\lceil \frac{C}{\eta_0^2}\cdot\left(\log(k/\varepsilon)+\log R\right)\right\rceil.
\]
Finally, for dyadic $w>w_{\mathrm{high}}$, define the large-weight per-round error
\[
\eta(w)\ :=\ \min\left\{\frac14,\ \sqrt{\frac{C\,(k\log( w/k)+\log R)}{w}}\right\}.
\]

\paragraph{Recursive construction.}
Let $U_0=[n]$. For rounds $r=0,1,\dots,R-1$, given $U_r$ we choose a puncturing set $I_r\subseteq U_r$,
set $V_r:=U_r\setminus I_r$, and then form $U_{r+1}$ by including each element of $V_r$
independently with probability $1/2$.

For dyadic\footnote{Here by dyadic, we mean $w = 2^j w_{\low}$ for $j = 0,1,...,\log (w_{\high} / w_{\low})$.} $w\in[w_{\low},w_{\high}]$, define
\[
\calF^{(r)}_{(w,2w]}
:=
\left\{\,\supp(x)\cap U_r:\ x\in\calC,\ \ w<|\supp(x)\cap U_r|\le 2w\,\right\}.
\]
Define the low-weight family
\[
\calF^{(r)}_{\le w_{\low}}
:=
\left\{\,\supp(x)\cap U_r:\ x\in\calC,\ \ |\supp(x)\cap U_r|\le w_{\low}\,\right\}.
\]
(All these families are $k$-moonflower-free by \Cref{lem:WSF_free_preserved_under_projection}.)

\smallskip
\begin{enumerate}[(a)]
    \item \emph{Low weights}: capture deterministically. Let
    \[
    I_{r,\mathrm{low}}:=\supp\!\left(\calF^{(r)}_{\le w_{\low}}\right)\subseteq U_r.
    \]
    By \Cref{lem:moonflower_free_family_have_small_support}, $|I_{r,\mathrm{low}}|\le k\,w_{\low}$.
    \item \emph{Medium weights}: puncture to shrink traces. Choose a sufficiently small absolute constant $\theta_0 > 0$, say $\theta_0 = 10^{-3}$ after adjusting constants in
    the Chernoff bound. 
    For each dyadic $w\in[w_{\low},w_{\high}]$, apply \Cref{thm:universe_reduction}
    to $\calF^{(r)}_{(w,2w]}$ with parameter $\lambda = \theta_0 \eta_0^2$ to obtain $I_{r,w}\subseteq U_r$
    such that
    \begin{equation}\label{eq:medium_trace_bound_v2}
    \big|\left(\calF^{(r)}_{(w,2w]}\right)_{\overline{I_{r,w}}}\big|
    \ \le\
    \exp (\theta_0 \eta_0^2 w),
    \end{equation}
    and
    \[
    |I_{r, w}| \leq \frac{k}{\eta_0^2} \cdot \poly \left(\log \frac{kw}{\eta_0} \right).
    \]
    Here we use the fact that $\eta_0 < 1$ and $\theta_0$ is an absolute constant, so its contribution is absorbed into the polylogarithmic factor.
    \item \emph{Define $I_r$ and recurse.} Let
    \[
    I_r\ :=\ I_{r,\mathrm{low}} \ \cup\ \bigcup_{\substack{\text{dyadic }w\\ w_{\low}\le w\le w_{\high}}} I_{r,w},
    \qquad
    V_r:=U_r\setminus I_r,
    \]
    and sample $U_{r+1}\subseteq V_r$ by keeping each coordinate with probability $1/2$.
    \item \emph{Output.} Define
    \[
    T\ :=\ \left(\bigcup_{r=0}^{R-1} I_r\right)\ \cup\ U_R,
    \qquad
    \alpha(i):=2^r\ \text{ if } i\in I_r,\quad \alpha(i):=2^R\ \text{ if } i\in U_R.
    \]
    (The sets $I_0,\dots,I_{R-1},U_R$ are disjoint since $U_{r+1}\subseteq U_r\setminus I_r$.)
\end{enumerate}

\paragraph{Successful Sampling.}
Fix a round $r$ and condition on the entire history up to round $r$ so that $U_r,I_r,V_r$ are fixed.

\smallskip
\emph{Medium regime.}
Fix dyadic $w\in[w_{\low},w_{\high}]$. Consider the trace family on $V_r$:
\[
\left(\calF^{(r)}_{(w,2w]}\right)_{V_r}.
\]
Since $V_r\subseteq U_r\setminus I_{r,w}$, restriction cannot increase size, so
\[
\big|\left(\calF^{(r)}_{(w,2w]}\right)_{V_r}\big|
\le
\big|\left(\calF^{(r)}_{(w,2w]}\right)_{\overline{I_{r,w}}}\big| \leq \exp (\theta_0 \eta_0^2 w).
\]
Moreover, for any $A$ in this family, we have $|A|\le 2w$.
Applying \Cref{lem:sparsif_chernoff_v2} with $\Delta:=\eta_0 w$ gives 
\[
\Pr_{U_{r+1}}\left[\big|\,2|A\cap U_{r+1}|-|A|\,\big|>\eta_0 w\right]
\le
2\exp\!\left(-\Omega(\eta_0^2 w)\right).
\]
Using~\eqref{eq:medium_trace_bound_v2} and the definition of $w_{\low}$,
the union bound over all $A\in (\calF^{(r)}_{(w,2w]})_{V_r}$ fails with probability at most
$\exp(-\Omega(\eta_0^2 w))\le 1/(100R\log(2w_{\high}))$ for appropriate choice of constants $C, \theta_0$.

\smallskip
\emph{Large regime.}
Fix dyadic $w>w_{\high}$ and consider the trace family
\[
\calH^{(r)}_{(w,2w]}\ :=\ \left(\calF^{(r)}_{(w,2w]}\right)_{V_r}.
\]
Since $\calF^{(r)}_{(w,2w]}$ is $k$-moonflower-free and $w\ge k$ in this regime,
\Cref{main_wsf_theorem} gives 
\[
|\calH^{(r)}_{(w,2w]}|
\ \le\
|\calF^{(r)}_{(w,2w]}|
\ \le\
\left(\frac{C_0 w}{k}\right)^k,
\]
for some constant $C_0 > 0$.
For any $A\in\calH^{(r)}_{(w,2w]}$ we have $|A|\le 2w$, so by \Cref{lem:sparsif_chernoff_v2} with
$\Delta:=\eta(w)\,w$, 
\[
\Pr_{U_{r+1}}\left[\big|\,2|A\cap U_{r+1}|-|A|\,\big|>\eta(w)\,w\right]
\le
2\exp\!\left(-\Omega(\eta(w)^2 w)\right)
=
2\exp\!\left(-\Omega(k\log(e w/k)+\log R)\right).
\]
For large enough $C$ in the definition of $\eta(w)$, the union bound over all
$A\in\calH^{(r)}_{(w,2w]}$ fails with probability at most $1/(100R\cdot 2^{j+2})$
when $w=2^j$. Summing over dyadic $w>w_{\high}$ gives failure probability at most $1/(100R)$.

\smallskip
Combining medium and large regimes and summing over the $O(\log w_{\high})$ medium scales,
we obtain that conditioned on the past, round $r$ fails with probability at most $1/(50R)$.
A union bound over $r=0,\dots,R-1$ yields that all rounds succeed with probability at least $49/50$.
Also $\E[|U_R|]\le n2^{-R}\le 1$, hence $\Pr[|U_R|\le 50]\ge 49/50$ by Markov.
Intersecting gives overall success probability at least $2/3$.

\paragraph{Accuracy for a fixed codeword.}
Fix a successful outcome as above. Fix a codeword $x\in\calC$ with $S:=\supp(x)$.

Let $\widehat{N}_r = \sum_{i=0}^{r-1} 2^i |S \cap I_i| + 2^r |S \cap U_r|$ denote the weight estimate for the weight of $x$ at round $r$. Note that $\widehat{N}_0 = |S|$ is the Hamming weight of $x$. Our goal is to show that $\widehat{N}_R = (1\pm \varepsilon) \widehat{N}_0$. 

Consider the next weight estimate  $$\widehat{N}_{r+1} = \sum_{i=0}^{r-1} 2^i  |S \cap I_i| + 2^r (|S \cap I_r| + 2 |S \cap U_{r+1}|).$$

Let $w_r := |S \cap U_r|$ be the true residual weight at round $r$, and let 

$$
\epsilon_r := \begin{cases} 0 & w_r \leq w_{\low},\\
\eta_0 & w_{\low} < w_r \leq w_{\high}\\
\eta(w_r) & w_r > w_{\high}
\end{cases}.
$$

Then, as we are in a successful outcome, we have,
$$\left|2 |S \cap U_{r+1}| - |S \cap (U_r\setminus I_r)| \right| \leq \epsilon_r w_r.$$

Thus, 
$$\left| 2 |S \cap U_{r+1}| + |S \cap I_r|  - |S \cap U_r| \,\right| \leq \epsilon_r w_r.$$

Combining the above equations, we get 
$$|\widehat{N}_{r+1} - \widehat{N}_r| = 2^r \cdot \left| 2 |S \cap U_{r+1}| + |S \cap I_r|  - |S \cap U_r| \,\right| \leq \epsilon_r 2^r w_r \leq \epsilon_r \widehat{N}_r.$$

Thus, we have 
$$(1-\epsilon_r) \widehat{N}_r \leq \widehat{N}_{r+1} \leq (1+\epsilon_r) \widehat{N_r}.$$

Applying the above inequality for $r = 0,\ldots,R-1$, we get 
$$\widehat{N}_0  \cdot \prod_{r=0}^{R-1} (1-\epsilon_r) \leq \widehat{N}_R \leq \widehat{N}_0  \cdot \prod_{r=0}^{R-1} (1+\epsilon_r).$$

Further, as $\epsilon_r < 1/2$, the above can be simplified to 
\begin{equation}\label{eq:sparse1}
    \widehat{N}_0  \cdot \exp\left(-2 \sum_{r=0}^{R-1}\epsilon_r\right) \leq \widehat{N}_R \leq \widehat{N}_0  \cdot \exp\left( \sum_{r=0}^{R-1}\epsilon_r\right) .
\end{equation}

It remains to bound $\sum_r \epsilon_r$. The number of rounds with $w_{\low}<w_r\le w_{\high}$ is $O(\log(2w_{\high}))$ since
$w_{r+1}\le (1+\eta_0)w_r/2\le (3/5)w_r$.
Thus $\sum_{w_{\low}<w_r\le w_{\high}}\epsilon_r\le O(\eta_0\log(2w_{\high}))\le \varepsilon/50$.

For rounds with $w_r>w_{\high}$, we have $\epsilon_r\le \eta(w_r)$ and 
$w_{r+1}\le (1+\epsilon_r)w_r/2\le (5/8)w_r$. A geometric-series estimate gives
\[
\sum_{w_r>w_{\high}}\epsilon_r
\ \le\
O\!\left(\sqrt{\frac{k\log(e w_{\high}/k)+\log R}{w_{\high}}}\right)
\ \le\ \varepsilon/50
\]
by the definition of $w_{\high}$ (with $C$ large enough).

Combining the above we get that $\sum_r \epsilon_r \leq \varepsilon/25$. Thus, in particular we must have 
$\widehat{N}_R = (1 \pm \varepsilon) N_0$, as we wanted.

\ignore{
Let $W_r:=|S\cap U_r|$ be the true residual weight at round $r$.  
Define the estimator recursively:
\[
\widehat W_R:=|S\cap U_R|,
\qquad
\widehat W_r:=|S\cap I_r|+2\widehat W_{r+1}\quad (r=R-1,\dots,0).
\]
Then $\widehat W_0=\widehat{\mathrm{wt}}_{T,\alpha}(x)$.

Let $R_r:=S\cap V_r$, so $W_r=|S\cap I_r|+|R_r|$ and $W_{r+1}=|R_r\cap U_{r+1}|$.
If $W_r\le w_{\min}$ then $S\cap U_r\subseteq I_{r,\mathrm{tiny}}\subseteq I_r$, so $R_r=\emptyset$
and $W_{r+1}=0$ and $\widehat W_r=W_r$ exactly.

Otherwise, let dyadic $w$ satisfy $W_r\in(w,2w]$.
If $w\le w_{\high}$ then by the medium-regime guarantee we have
\[
\big|\,2W_{r+1}-|R_r|\,\big|\ \le\ \eta_0 w\ \le\ \eta_0 W_r.
\]
If $w>w_{\high}$ then by the large-regime guarantee we have
\[
\big|\,2W_{r+1}-|R_r|\,\big|\ \le\ \eta(w)\,w\ \le\ \eta(w)\,W_r.
\]
In either case we obtain
\begin{equation}\label{eq:eta_r_def_v2}
\big|\,2W_{r+1}-|R_r|\,\big|\ \le\ \eta_r\,W_r,
\qquad\text{where}\qquad
\eta_r:=\begin{cases}
0 & W_r\le w_{\min},\\
\eta_0 & w_{\min}<W_r\le w_{\high},\\
\eta(w) & W_r>w_{\high}.
\end{cases}
\end{equation}

\paragraph{Error propagation.}
Define $\xi_R:=0$ and for $r=R-1,\dots,0$ set
\[
\xi_r\ :=\ \eta_r+(1+\eta_r)\xi_{r+1}.
\]
We claim that for all $r$,
\begin{equation}\label{eq:xi_claim_v2}
(1-\xi_r)W_r\ \le\ \widehat W_r\ \le\ (1+\xi_r)W_r.
\end{equation}
We prove this by backward induction on $r$.
The base case $r=R$ holds since $\widehat W_R=W_R$ and $\xi_R=0$.

Assume~\eqref{eq:xi_claim_v2} holds for $r+1$. Then
\[
\widehat W_r
=
|S\cap I_r|+2\widehat W_{r+1}
\le
|S\cap I_r|+2(1+\xi_{r+1})W_{r+1}.
\]
Using $2W_{r+1}=|R_r|+E_r$ where $|E_r|\le \eta_r W_r$ from~\eqref{eq:eta_r_def_v2}, and
$2W_{r+1}\le |R_r|+\eta_r W_r\le (1+\eta_r)W_r$, we get $2W_{r+1}\le (1+\eta_r)W_r$ and hence
$2\xi_{r+1}W_{r+1}\le (1+\eta_r)\xi_{r+1}W_r$.
Therefore
\[
\widehat W_r
\le
|S\cap I_r|+|R_r|+\eta_r W_r+(1+\eta_r)\xi_{r+1}W_r
=
W_r+\xi_r W_r
=
(1+\xi_r)W_r.
\]
A symmetric argument gives the lower bound:
\begin{align*}
\widehat W_r
&   =
|S\cap I_r|+2\widehat W_{r+1}
\ge
|S\cap I_r|+2(1-\xi_{r+1})W_{r+1}
=
|S\cap I_r|+|R_r|-E_r-2\xi_{r+1}W_{r+1}
\\
& \ge
W_r-\eta_r W_r-(1+\eta_r)\xi_{r+1}W_r
=
(1-\xi_r)W_r.
\end{align*}

This proves~\eqref{eq:xi_claim_v2} for all $r$, in particular
$\widehat W_0\in(1\pm \xi_0)W_0$.

Finally, since $1+\xi_r=(1+\eta_r)(1+\xi_{r+1})$, we have
\[
1+\xi_0=\prod_{r=0}^{R-1}(1+\eta_r)\ \le\ \exp\!\left(\sum_{r=0}^{R-1}\eta_r\right),
\]
hence $\xi_0\le \exp(\sum_r\eta_r)-1$.

\paragraph{Bounding $\sum_r\eta_r$.}
The number of rounds with $w_{\min}<W_r\le w_{\high}$ is $O(\log(2w_{\high}))$ since
$W_{r+1}\le (1+\eta_0)W_r/2\le (3/5)W_r$.
Thus $\sum_{w_{\min}<W_r\le w_{\high}}\eta_r\le O(\eta_0\log(2w_{\high}))\le \varepsilon/50$.

For rounds with $W_r>w_{\high}$, we have $\eta_r\le \eta(w_r)$ with $w_r\asymp W_r$ dyadic and
$W_{r+1}\le (1+\eta_r)W_r/2\le (5/8)W_r$. A geometric-series estimate gives
\[
\sum_{W_r>w_{\high}}\eta_r
\ \le\
O\!\left(\sqrt{\frac{k\log(e w_{\high}/k)+\log R}{w_{\high}}}\right)
\ \le\ \varepsilon/50
\]
by the definition of $w_{\high}$ (choosing $C$ large enough).
Therefore $\sum_r\eta_r\le \varepsilon/25$, and for $\varepsilon\le 1/4$ we get
$\xi_0\le e^{\varepsilon/25}-1\le \varepsilon$ (after increasing constants).
Hence $\widehat{\mathrm{wt}}_{T,\alpha}(x)=\widehat W_0\in (1\pm\varepsilon)W_0$ for all $x\in\calC$.}

\paragraph{Size bound.}
We have $|T|\le \sum_{r=0}^{R-1}|I_r|+|U_R|$ and $|U_R|\le 50$ on the good event.
Also $|I_{r,\mathrm{low}}|\le k w_{\low}$.

For each dyadic $w\in[w_{\low},w_{\high}]$, apply
\Cref{thm:universe_reduction} to the layer family
$\calF^{(r)}_{(w,2w]}$ with $\lambda = \theta_0 \eta_0^2$ to obtain
\[
|I_{r,w}|
\ \le\
\frac{k}{\eta_0^2}\cdot \poly\!\left(\log(kw/\eta_0)\right).
\]
Since $w\le w_{\high}$ throughout, we have $\log w\le \log w_{\high}$ and
\[
\log(k/\eta_0)=\log(k/\varepsilon)+O(\log\log w_{\high}),
\]
hence
\[
|I_{r,w}|
\ \le\
\frac{k}{\eta_0^2}\cdot \poly\!\left(\log(k/\varepsilon),\log\log n\right),
\]
using $\log w_{\high} = O(\log(k/\varepsilon)+\log\log n)$ by the definition of $w_{\high}$.

There are $O(\log(2w_{\high}))$ such dyadic values. Thus
\[
|I_r|
\ \le\
\frac{k}{\eta_0^2}\cdot \poly\!\left(\log(k/\varepsilon),\log\log n\right),
\]
since $\eta_0^{-2}=\Theta(\log^2(2w_{\high})/\varepsilon^2)$ and
$w_{\low}=\Theta(\eta_0^{-2}(\log(k/\varepsilon)+\log R))$.
Finally $R=\Theta(\log n)$ gives
\[
|T|
\ \le\
\frac{k\log n}{\varepsilon^2}\cdot \poly\!\left(\log(k/\varepsilon),\log\log n\right),
\]
as claimed.

This finishes the proof of the full sparsification theorem.
\end{proof}

\subsection{Lower bound}\label{sec:sparsification_lower_bound}
In this subsection, we state and prove our lower bound construction. 

\begin{lemma}[\Cref{main:sparsification_lower_bound}, restated]
Let $k \ge 1$ and $\varepsilon \in (0,1)$. Then, for all large enough $n$, there exists an explicit $\calC \subseteq \{0, 1\}^n$ with $\mathrm{NRD}(\calC)= k$ such that any $\varepsilon$-sparsifier $(T, \alpha)$ of $\calC$ must satisfy
    \[
    |T| =\Omega \left( \frac{k \log (n/k)}{\varepsilon} \right).
    \]
\end{lemma}

\begin{proof}
We will assume $n$ is chosen large enough so that $\frac{k \log(n/k)}{\varepsilon} = O(n)$.
In the construction we identify $\{0,1\}^n$ with subsets of $[n]$.
Assume without loss of generality that $k$ divides $n$ and let $m=n/k$. Let $a_1 \le \cdots \le a_s$ be a maximal collection of integers such that $1 \le a_1$, $a_s < m$ and $a_{j+1} > (1+\varepsilon) a_j/(1-\epsilon)$ for all $j<s$. Clearly, such numbers exist for $s=\Omega(\log (n/k)/\varepsilon)$ (here is where we use the assumption that $n$ is large enough, as we clearly have $s \le n/k$). For two integers $i<j$ let $[i:j]=\{i,i+1,\ldots,j\}$. For $i \in [k], j \in [s]$ define the set
$$
S_{i,j} = [(i-1)m:(i-1)m+a_j].
$$
We take $\calC_i := \{S_{i,j}:  j \in [s]\}$ and $\calC := \cup_{i \in [k]} \calC_i$.

Observe that $\calC_i$ are defined on disjoint ground sets and that each $\calC_i$ is a chain. This implies that $\mathrm{NRD}(\calC_i)=1$ and hence \footnote{A size $k \times k$ permutation matrix is formed by taking the last coordinate of each $\calC_i$, i.e., $(i-1)m + a_s$'s as the diagonal.} $\mathrm{NRD}(\calC)=k$. Finally, let $(T,\alpha)$ be an $\varepsilon$-sparsifier for $\calC$. We claim that $|T| \ge |\calC|$ which concludes the proof.

Assume this is not the case. Then there must exist some $i \in [k]$ such that $|T \cap \Supp (\calC_i)| < |\calC| / k = s$. In particular, there must exist $i \in [k]$ and $j<s$ such that 
$$
T \cap S_{i,j} = T \cap S_{i,j+1}.
$$
This however cannot be the case as by construction $|S_{i,j+1}| > (1+\varepsilon) |S_{i,j}|$ and $(T,\alpha)$ is an $\varepsilon$-sparsifier for $\calC$.

\end{proof}

\ignore{
\begin{remark}[Where $\log\log n$ appears]
The $\log\log n$ dependence enters through $\log R$ in (i) the target failure probability
per round, and (ii) the choice of $w_{\high}$ and $w_{\min}$ to make union bounds hold uniformly
over all $R=\Theta(\log n)$ rounds when $k$ is small.
This does not affect the main point that the dependence on $n$ is a single $\log n$ factor.
\end{remark}}

%% file: appendix.tex
In this section, we prove \Cref{thm:frankl}, restated below for convenience.

\begin{theorem}[\Cref{thm:frankl} restated, \cite{Frankl1982ExtremalProblem}]\label{thn:frankl_restated}
    Let $\calA = \{A_1,\ldots,A_m\}$ be a family of $r$-sets and $\calB = \{B_1,\ldots,B_m\}$ be a family of $s$-sets such that (1) $A_i \cap B_i = \emptyset$ for $i = 1,2,\ldots,m$ and (2)  $A_i \cap B_j \neq \emptyset$ for $1 \leq i < j \leq m$. Then 
    \[
    m \leq \binom{r+s}{s}.
    \]
\end{theorem}

The main idea behind the proof of \Cref{thm:frankl} is as follows: we associate to each set $A_i$ a vector $u_i$ in some vectors space $W$ and show that $u_1,\ldots,u_m$ are linearly independent. Given this, it suffices to choose $W$ such that $\dim W = \binom{r+s}{s}$. In \cite{Frankl1982ExtremalProblem}, Frankl chose $W = \Sym^r (V)$ for some vector space $V$ of dimension $s+1$ where $\Sym^r (V)$ denotes the $r$-th symmetric tensor power of $V$.
 
\begin{definition}[Symmetric tensor product] Let $v_1,\ldots,v_d \in V \subseteq \rr^n$. The $d$-th symmetric tensor product of $v_1,\ldots,v_d$ denoted $v_1 \vee v_2 \vee \cdots \vee v_d$ is defined as
\[
v_1 \vee v_2 \vee \cdots \vee v_d := \frac{1}{d!}\sum_{\sigma \in S_d}v_{\sigma (1)} \otimes v_{\sigma (2)} \otimes \cdots \otimes v_{\sigma (d)}.
\]
In other words, the symmetric tensor product is the average of tensor products over all permutations.    
\end{definition}

\begin{definition}[Symmetric tensor power]
    Let $V \subseteq \rr^n$. The $d$-th tensor power of $V$ denoted $\Sym^d (V)$ is defined as:
    \[
    \Sym^d (V):= span\{v_1 \vee v_2 \vee \cdots \vee v_d: v_i \in V, \ \forall i \in [d]\}.
    \]
\end{definition}

\begin{fact}[Dimension of symmetric tensor power]\label{fact:dim_sym_tensor_power}
    $\Sym^d (V)$ is a vector space and
    \[
    \dim (\Sym^d (V)) = \binom{\dim (V) + d - 1}{d}.
    \]
\end{fact}

We say $T$ is a \emph{$w$-uniform set} if $|T| = w$. We are now ready to prove \Cref{thm:frankl}.
\begin{proof}(of \Cref{thn:frankl_restated}, \cite{Frankl1982ExtremalProblem})
    Adding distinct elements and enlarging the universe if necessary, we may assume $\calA$ is a and $\calB$ are families of $r$-uniform and $s$-uniform sets respectively. Let $X:= (\bigcup_i A_i) \cup (\bigcup_j B_j)$ denote the universe. Choose a vector space $V$ of dimension $s+1$. Assign to each $x \in X$ a vector $v_x \in V$ such that any $s+1$ of these vectors are linearly independent\footnote{For instance, we can choose distinct real numbers $t_x$ and set $v_x = (1, t_x,\ldots,t_x^{s}) \in \rr^{s+1}$. Then any $s+1$ vectors forms a Vandermonde matrix and are thus linearly independent.}. For $A_i = \{a_{i, 1},\ldots,a_{i, r}\}$, define $u_i$ to be the associated symmetric $r$-th tensor product:
    \[
    u_i := v_{a_{i, 1}} \vee v_{a_{i, 2}} \vee \cdots \vee v_{a_{i, r}}.
    \]
    Notice that for each $j$, the vectors $\{v_{b}: b \in B_j\}$ span an $s$-dimensionl subspace of $V$ as $|B_j| = s$. Since $\dim (V) = s + 1$, there exists a nonzero linear functional $f_j: V \to \rr$ such that $f_j (v_b) = 0$ for all $v_b \in B_j$. We can extend $f_j$ to obtain a functional $F_j :\Sym^r (V) \to \rr$ defined as
    \[
    F_j (v_{x_1} \vee v_{x_2} \vee \cdots \vee v_{x_r}) = f_j(v_{x_1})f_j(v_{x_2}) \cdots f_j(v_{x_r}).
    \]
    In particular, this implies $F_j(u_i) = \prod_{a \in A_i}f_j(v_a)$. Now, if $A_i \cap B_j \neq \emptyset$, then there exists some $a \in A_i \cap B_j$ which implies $f_j(v_a) = 0$ and hence $F_j (u_i) = 0$. On the other hand, if $A_i \cap B_j = \emptyset$ then for all $a \in A_i$, we have $|B_j \cup \{a\}| = s+1$. By the construction of $V$, we know $\{v_b: b \in B_j\} \cup \{a\}$ is linearly independent for all $a \in A_i$. This gives $f_j(v_a) \neq 0$ for all $a \in A_i$ since $f_j$ vanishes only on $\mathrm{span} (\{v_b: v \in B_j\})$. Hence, in this case we have $F_j(u_i) \neq 0$.

    Together with the assumptions on the set families $\calA$ and $\calB$, we thus know (1) $F_i (u_i) \neq 0$ for all $i = 1,2,\ldots,m$ and (2) $F_j (u_i) = 0$ for all $1 \leq i < j \leq m$. Using this, we can show $u_1,\ldots,u_m$ are linearly independent.

    To see this, suppose for contradiction $u_1,\ldots,u_m$ are linearly dependent. Then there exists $c_1,\ldots,c_m$ not all zeros such that
    \[
    c_1u_1 + \cdots + c_m u_m = 0.
    \]
    Let $k \in [m]$ denote the largest index such that $c_k \neq 0$. Then we know $c_1u_1 + \cdots + c_k u_k = 0$. Applying $F_k$ to both sides, we have by linearity that 
    \[
    c_1 F_k(u_1) + \cdots + c_k F_k(u_k) = 0.
    \]
    Since $F_k(u_i) = 0$ for all $i < k$, the above equality is equivalent to $c_k F_k(i_k) = 0$ which implies $c_k = 0$ as $F_k (u_k) \neq 0$. This is a contradiction.

    Finally, since $u_1,\ldots,u_m$ are linearly independent, we know
    \[
    m \leq \dim (\Sym^r (V)) \leq \binom{\dim (V) + r - 1}{r} = \binom{r+s}{s}
    \]
    as desired where the last equality follows from \Cref{fact:dim_sym_tensor_power}.
\end{proof}